\documentclass[11pt,center]{amsart}
\usepackage{amsmath,amssymb}

\usepackage{hyperref}

\usepackage{tikz}
\usetikzlibrary{cd,arrows.meta,calc,decorations.markings,patterns,intersections}

\usepackage[bb=ams,scr=euler]{mathalpha}
\usepackage{enumitem}
\setenumerate{label=(\arabic*)}

\usepackage{thmtools}
\declaretheorem[name=Theorem]{theorem}
\declaretheorem[name=Lemma,numberwithin=section]{lemma}
\declaretheorem[name=Claim,sibling=lemma]{claim}
\declaretheorem[name=Proposition,sibling=lemma]{proposition}
\declaretheorem[name=Corollary,sibling=theorem]{corollary}

\declaretheorem[name=Remark,sibling=lemma]{remark}

\declaretheorem[name=Definition,sibling=lemma]{definition}
\declaretheorem[name=Definition,sibling=theorem]{bigdefinition}

\makeatletter
\newcommand{\abs}[1]{\left|#1\right|}
\newcommand{\bd}{\partial}
\newcommand{\C}{\mathbb{C}}
\renewcommand{\d}{\mathit{d}}
\newcommand{\id}{\mathit{id}}

\newcommand{\R}{\mathbb{R}}

\newcommand{\set}[1]{\left\{#1\right\}}
\newcommand{\Z}{\mathbb{Z}}
\newcommand{\Q}{\mathbb{Q}}

\def\@secnumfont{\bfseries}
\renewcommand\section{\@startsection{section}{1}{0pt}{-3.5ex \@plus -1ex \@minus -.2ex}{2.3ex \@plus.2ex}{\centering\bfseries}}
\renewcommand{\subsection}{\@startsection{subsection}{2}\z@{.5\linespacing\@plus.7\linespacing}{-.5em}{\normalfont\bfseries}}
\makeatother

\date{\today}
\author{Filip Broćić}
\email{filipvbrocic@gmail.com}
\address{Lehrstuhl für Analysis und Geometrie, Universität Augsburg, Universitätsstrasse 14, 86159 Augsburg, Germany}

\author{Dylan Cant}
\email{dylan@dylancant.ca}
\address{Département de mathématiques et de statistique, Université de Montréal, Pavillon André-Aisenstadt, 2920 Chemin de la Tour, Montréal, Québec, Canada}

\begin{document}
\title{Lengths of Reeb chords and Viterbo restriction}
\begin{abstract}
  Let $\Lambda$ be a Legendrian in the contact boundary of a Liouville domain $\Omega$. We explain how the non-existence of Reeb chords with endpoints on $\Lambda$ of length up to $a$ enables one to embed $D_{\epsilon}T^{*}\Lambda\times D(a)$ into $\Omega$ in an exact way. As in earlier work of Zhengyi Zhou, we use the Viterbo restriction map to deduce a contradiction in certain cases. In particular, we show that if $M$ admits a submersion from a product of spheres (e.g., the $n$-torus), then all compact Legendrians $\Lambda\subset ST^{*}M$ admit a Reeb chord for every choice of contact form $ST^{*}M$. The obstruction we use in this case is based on the idea of inverting the degree-$n$ classes in cohomology, and is similar to the notion of string point invertibility introduced by Egor Shelukhin.
\end{abstract}
\maketitle

\section{Introduction}
\label{sec:introduction}

\subsection{The notion of
  $n$-invertibility}
\label{sec:lengths-reeb-chords}

Suppose that $\Omega^{2n}$ is a Liouville
domain and consider a Legendrian submanifold
$\Lambda\subset\bd \Omega$. The primary goal
of this paper is to bound the length of the
shortest non-constant Reeb chord of $\Lambda$
from above. Here the Reeb flow is associated
to the canonical contact form on $\bd\Omega$.

Let $W$ denote the Liouville manifold obtained by completion of $\Omega$. Our main results are phrased in terms of the
symplectic cohomology $\mathit{SH}(W)$,
and its additional structures:
\begin{enumerate}
\item\label{struct-1}
  $\mathit{PSS}:H^{*}(W)\to
  \mathit{SH}(W)$,
\item\label{struct-2} BV operator
  $\Delta:\mathit{SH}(W)\to
  \mathit{SH}(W)$,
\item\label{struct-3} pair-of-pants
  product
  $\ast:\mathit{SH}(W)\otimes
  \mathit{SH}(W)\to
  \mathit{SH}(W)$.
\end{enumerate}
The symplectic cohomology
$\mathit{SH}(W)$ and its three
additional structures are well-known; see,
e.g.,
\cite{viterbo-GAFA-1999,viterbo-arXiv-1996,seidel-IP-2008,abouzaid-EMS-2015,zhao-thesis-2016,zhou-jtopol-2021,zhou-JSG-2022,zhou-adv-math-2022}. For
further details on the conventions we use, see
\cite{brocic-cant-arXiv-2025}.

This leads us to our main definition:
\begin{bigdefinition}\label{definition:n-invertible}
  A Liouville manifold $W$ is
  \emph{$n$-invertible} if the only
  $\Delta$-invariant ideal of
  $\mathit{SH}(W)$ which contains the
  subspace $\mathit{PSS} (H^{n}(W))$ is
  equal to $\mathit{SH}(W)$.
\end{bigdefinition}
Here we comment that $n$-invertible domains
always have dimension $2n$, i.e., we adopt
the somewhat common convention in symplectic
geometry that $n$ stands for ``half of the
dimension.'' The definition is closely related to that of \emph{string point invertibility} introduced in \cite{shelukhin-GAFA-2022}.

Let us set some terminology before proceeding
to our main theorem.
\begin{itemize}
\item The domain $\Omega$ is completed to a
  \emph{Liouville manifold} $W$, which means
  that $\Omega\subset W$ is embedded in such
  a way that the Liouville flow lines passing
  through the contact boundary $\Omega$ define the
  symplectization end $SY$ of $W$ (here $Y\simeq \partial \Omega$ is the ideal boundary).
\item The \emph{radial function} $r$ is
  $1$-homogeneous $\d r(Z)=r$ and satisfies
  $r|_{\partial \Omega}=1$. This function is
  smooth on the symplectization end $SY$
  and extends continuously to the rest of
  $\Omega$.
\item The \emph{Reeb vector field} is defined to be $R=X_{r}$ (defined on $SY$).
\item The \emph{Floer cohomology group} $\mathit{SH}_{c}(\Omega)$ is defined as the ``category theoretic limit'' of Floer cohomologies of Hamiltonian systems whose generating vector field agree with $cR$ outside of a compact set. We restrict $c$ to be in the complement of the \emph{spectrum} $\mathit{Spec}(\Omega)$ of periods of $R$. We assume throughout the body of the paper that Floer cohomologies are computed using coefficients in a field of characteristic two.
\end{itemize}

The spaces $\mathit{SH}_{c}(\Omega)$, for $c$ in the
complement of the spectrum, form a
\emph{persistence module} (in the sense of
\cite{polterovich_shelukhin_persistence_1});
the colimit of this persistence module is
easily seen to be the symplectic cohomology
$\mathit{SH}(W)$.

The three structures
$\mathit{PSS},\Delta,\ast$ can be lifted to
the persistence module as:
\begin{enumerate}
\item natural maps
  $\mathit{PSS}:H^{*}(W)\to \mathit{SH}_{c}(\Omega)$ for
  $c>0$,
\item natural endomorphisms
  $\Delta:\mathit{SH}_{c}(\Omega)\to
  \mathit{SH}_{c}(\Omega)$,
\item natural products
  $\ast:\mathit{SH}_{c_{1}}(\Omega)\otimes
  \mathit{SH}_{c_{2}}(\Omega)\to
  \mathit{SH}_{c_{1}+c_{2}}(\Omega)$.
\end{enumerate}

We now attempt to ``quantify'' the notion of
$n$-invertibility; first set:
\begin{equation*}
  \mathscr{I}_{0,c}=\mathrm{span}\set{ \mathit{PSS}(H^{n}(\Omega))\ast \mathit{SH}_{c'}(\Omega):c'<c},
\end{equation*}
where $\mathit{PSS}(H^{n}(\Omega))$ takes values in $\mathit{SH}_{c-c'}(\Omega)$.
Similarly define for $k\ge 1$:
\begin{equation*}
  \mathscr{I}_{k,c}=\mathscr{I}_{k-1,c}+\mathrm{span}\set{\Delta_{c_{1}}(\mathscr{I}_{k-1,c_{1}})\ast \mathit{SH}_{c_{2}}:c_{i}>0\text{ and }c_{1}+c_{2}=c},
\end{equation*}
and let
$\mathscr{I}_{c}\subset \mathit{SH}_{c}(\Omega)$ be the
union of $\mathscr{I}_{k,c}$ over all
$k$. Note that, since $\mathit{SH}_{c}(\Omega)$ is a
finite dimensional vector space, the colimit
$\mathscr{I}_{c}=\mathscr{I}_{c,k}$ is
attained for some finite number $k$. This
subspace
$\mathscr{I}_{c}\subset \mathit{SH}_{c}(\Omega)$ is a
natural subspace of the persistence module
(i.e., continuation maps send
$\mathscr{I}_{c'}$ into $\mathscr{I}_{c}$ for
$c'<c$), and so the colimit of
$\mathscr{I}_{c}$ is a subspace of the
symplectic cohomology
$\mathscr{I}\subset \mathit{SH}(W)$.

\begin{lemma}
  The ideal $\mathscr{I}$ is the smallest
  $\Delta$-invariant ideal of
  $\mathit{SH}(W)$ which contains the
  subspace $\mathit{PSS}(H^{n}(W))$.
\end{lemma}
\begin{proof}
  Define ideals
  $\mathscr{I}_{0}=
  \mathit{PSS}(H^{n}(W))\ast
  \mathit{SH}(W)$, and:
  \begin{equation*}
    \mathscr{I}_{k}=\mathscr{I}_{k-1}+\Delta(\mathscr{I}_{k-1})\ast \mathit{SH}(W),
  \end{equation*}
  and then $\mathscr{I}_{k}$ is easily seen
  to be the colimit of $\mathscr{I}_{c,k}$ as $c\to\infty$,
  and since colimits commute with colimits,
  the monotone union
  $\mathscr{I}_{0}\cup \mathscr{I}_{1}\cup
  \dots$ is equal to the colimit
  $\mathscr{I}$. From this description of
  $\mathscr{I}$ we conclude the desired
  result.
\end{proof}

This leads to our second main definition in
terms of the persistence refinement of the
ideal $\mathscr{I}$:
\begin{bigdefinition}
  The $n$-invertibility capacity of a
  Liouville domain $\Omega$ is:
  \begin{equation*}
    c_{\mathrm{ni}}(\Omega)=\inf\set{c>0:\mathit{PSS}(1)\in \mathscr{I}_{c}}.
  \end{equation*}
  Then $c_{\mathrm{ni}}(\Omega)<\infty$ if
  and only if the completion $W$ is $n$-invertible.
\end{bigdefinition}

Our main result is:
\begin{theorem}\label{theorem:main}
  Every Legendrian $\Lambda\subset \bd\Omega$
  bounds a non-constant Reeb chord with
  period at most $c_{\mathrm{ni}}(\Omega)$, provided $c_{\mathrm{ni}}(\Omega)<\infty$.
\end{theorem}

\subsection{Examples of $n$-invertible manifolds}

\subsubsection{Domains with vanishing symplectic cohomology}
\label{sec:vanish-sympl-cohom}

If $\mathit{SH}(W)=0$, then the conclusion of Theorem \ref{theorem:main} holds since the condition from Definition \ref{definition:n-invertible} is trivially satisfied. This recovers a result of Zhou \cite{zhou-JSG-2022}.

\subsubsection{Cotangent bundles of products of spheres}
\label{sec:cotang-bundl-spher}

\begin{theorem}\label{theorem:spheres}
  For $n\ge 1$ we construct classes:
  \begin{itemize}
  \item $A\in H_{n}(\Lambda S^{n})$ and,
  \item $B \in H_{2n-1}(\Lambda S^n)$,
  \end{itemize}
  such that: $A\ast \Delta( B \ast [\mathit{pt}] ) = [S^n]$. More generally, if $M=S^{n_{1}}\times \dots \times S^{n_{k}}$ with $n=n_{1}+\dots+n_{k}$, there are classes:
  \begin{equation*}
    A_i\in H_{n}(\Lambda M)\quad B_{i} \in H_{n+n_{i}-1}(\Lambda M)\text{ for }i=1,\dots,k
  \end{equation*}
  such that the sequence $C_{i}$, $i=0,\dots,k$, defined by $C_{0}=[\mathrm{pt}]$:
  \begin{equation*}
    C_{i}=A_i \ast \Delta(B_{i}\ast C_{i-1})\text{ for }i=1,\dots,k
  \end{equation*}
  satisfies $C_{k}=[S^{n_{1}}\times \dots \times S^{n_{k}}]$.
\end{theorem}
This proves that $T^{*}(S^{n_{1}}\times \dots \times S^{n_{k}})$ is $n$-invertible,\footnote{For odd-dimensional spheres, one can appeal to formulas in \cite[Theorem 16]{menichi-CMH-2009}.} due to the morphism of BV-algebras established in \cite{viterbo-GAFA-1999,viterbo-arXiv-1996,abbondandolo-schwarz-CPAM-2006,abouzaid-EMS-2015,brocic-cant-arXiv-2025}:
\begin{equation*}
  H_{*}(\Lambda M)\to \mathit{SH}(T^{*}M)
\end{equation*}
sending $[\mathit{pt}]$ to $\mathit{PSS}(F)$ where $F$ is the cohomology class represented by the cotangent fiber. Moreover, our argument also gives estimates on the $n$-invertibility capacity in terms of the lengths of the loops appearing in the classes $A,B_{1},\dots,B_{k}$, where the length functional is the one determined by the domain $\Omega\subset T^{*}M$ considered in \cite{brocic-cant-arXiv-2025}.

\subsection{Rationality constants of non-exact Lagrangians}
\label{sec:rationality-intro}

Let $\Omega$ be a Liouville domain. Recently, there has been interest \cite{zhou-JSG-2022,atallah-chasse-leclercq-shelukhin-arXiv-2024} in bounding on the \emph{rationality constant} of a non-exact Lagrangian inside $\Omega$. Here we recall that a Lagrangian $L\subset \Omega$ determines a Liouville class $[\lambda|_{L}]\in H^{1}_{\mathrm{dR}}(L)$, and the \emph{rationality constant} $\rho(L)$ is the infimal positive period of this class.

Our second main result is the following bound on the rationality constants of \emph{aspherical}\footnote{A manifold $L$ is said to be aspherical if its universal cover is contractible.} Lagrangians inside of $T^{*}T^{n}$:
\begin{theorem}\label{theorem:rationality}
  Let $\Omega\subset T^{*}T^{n}$ be the unit codisk bundle. Then there exists a constant $C>0$ such that if $L \subset \Omega$ is a closed, non-exact, Lagrangian which is aspherical then $\rho(L)\le C$.
\end{theorem}
It is noteworthy that our result applies to $T^{*}T^{n}$, whereas previous results concerning aspherical Lagrangians $L$ such as \cite{zhou-JSG-2022,li-yin-arXiv-2023} do not apply when the ambient space is $T^{*}T^{n}$ (but rather apply when the ambient space behaves like the cotangent bundle of a non-aspherical manifold).

\begin{remark}\label{remark:on-the-constant-C}
  The constant $C$ can be chosen to be any number $2c$ such that there exist classes:
  \begin{equation*}
    A_{i}^{\pm}\in \mathit{SH}_{c_{i,\pm}}(\Omega)\text{ for }c_{i,\pm}>0\text{ with }c_{i,+}+c_{i,-}\le c
  \end{equation*}
  such that:
  \begin{itemize}
  \item $A_{i}^{\pm}$ lies in the free homotopy class of the loop $t\mapsto x\pm te_{i}$,
  \item $A_{i}^{+}\ast A_{i}^{-}=1$ holds in $\mathit{SH}_{c}(\Omega)$.
  \end{itemize}
  For related bounds, see \cite{brocic-cant-arXiv-2025}.
\end{remark}

\subsection{Methods of proof}
\label{sec:methods-of-proof}

The first step used in the proof of Theorem \ref{theorem:main} is a variation of the geometric construction similar to \cite{mohnke-annals-2001} and \cite[Proposition 4.4]{opshtein-arXiv-2025}. The result can be considered as a version of the $1$-jet neighborhood theorem for Legendrians.

\begin{theorem}\label{theorem:geometric-construction}
  Let $\Omega$ be a Liouville domain inducing a radial coordinate $r$ on $W$, and let $\alpha=\lambda/r$ be the resulting contact form on the ideal boundary $Y$ of the completion $W$. Let $\Lambda_{t}\subset Y$ be a Legendrian isotopy, and suppose that:
  \begin{itemize}
  \item for each smooth path $y(t)\in \Lambda_{t}$ it holds that $y^{*}\alpha>a\d t$ ,
  \item $\Lambda_{t}\cap \Lambda_{t'}=\emptyset$ for each $t\ne t'$ in $[0,1]$.
  \end{itemize}
  Then there is an exact embedding:
  \begin{equation*}
    D_{\epsilon}T^{*}\Lambda\times D(a)\to \Omega
  \end{equation*}
  taking values in the interior of $\Omega$. Here $D_{\epsilon}T^{*}\Lambda$ is an arbitrarily small disk bundle in $T^{*}\Lambda$, $D(a)$ is the closed disk of area $a$, and $D_{\epsilon}T^{*}\Lambda\times D(a)$ is equipped with its standard Liouville structure.
\end{theorem}

It is clear that if $\Lambda$ has no Reeb chords for the contact form $\alpha$ (as in the statement of Theorem \ref{theorem:geometric-construction}) of periods up to $a$, then we can find an isotopy $\Lambda_{t}$ satisfying the hypotheses of Theorem \ref{theorem:geometric-construction}, and thereby conclude an exact embedding $D_{\epsilon}T^{*}\Lambda\times D(a)\to \Omega$.

The next step is to appeal to the well-known \emph{Viterbo restriction} (or \emph{Viterbo transfer}) map, whose properties are summarized in the following theorem:

\begin{theorem}\label{theorem:VR-main}
  If $K,\Omega$ are smooth Liouville domains and $\iota:K\to \Omega$ is an exact embedding taking values in the interior of $\Omega$, then, for each pair of numbers $0<c<b$ such that $c\not\in \mathit{Spec}(\Omega)$ and $b\not\in \mathit{Spec}(K)$, there is a linear map:
  \begin{equation*}
    \mathfrak{R}:\mathit{SH}_{c}(\Omega)\to \mathit{SH}_{b}(K).
  \end{equation*}
  Moreover, this map respects structures:
  \begin{itemize}
  \item $\mathfrak{R}\circ \Delta=\Delta\circ \mathfrak{R}$,
  \item $\mathfrak{R}(\zeta_{1}\ast \zeta_{2})=\mathfrak{R}(\zeta_{1})\ast \mathfrak{R}(\zeta_{2})$,
  \item $\mathfrak{R}\circ \mathit{PSS}=\mathit{PSS}\circ \mathfrak{R}$,
  \end{itemize}
  where $\mathfrak{R}:H^{*}(\Omega)\to H^{*}(K)$ is the pullback on cohomology in the last item.
\end{theorem}
Thus, given a Legendrian $\Lambda$ which has no chords of length up to $a$, every smooth Liouville subdomain $K\subset D_{\epsilon}T^{*}\Lambda \times D(a)$ receives a map:
\begin{equation*}
  \mathfrak{R}:\mathit{SH}_{c}(\Omega)\to \mathit{SH}_{b}(K)
\end{equation*}
compatible with structures, as in Theorem \ref{theorem:VR-main}. In particular, if $a$ is larger than $c_{\mathrm{ni}}(\Omega)$, then we can pick $c_{\mathrm{ni}}(\Omega)<c<b<a$ so that:
\begin{equation*}
  1=0\in \mathit{SH}_{b}(K).
\end{equation*}
This is because $1$ lies in $\mathscr{I}_{c}(\Omega)$, and $\mathfrak{R}(\mathscr{I}_{c}(\Omega))$ vanishes in $\mathit{SH}_{b}(K)$ since:
\begin{equation*}
  \mathfrak{R}:H^{n}(\Omega)\to H^{n}(K)
\end{equation*}
vanishes (as there are no non-trivial degree $n$ classes in $D_{\epsilon}T^{*}\Lambda\times D(a)$).

The next step uses the following filtered Künneth map result:
\begin{theorem}\label{theorem:filtered-kunneth}
  Let $Q\times D(a)$ be a Cartesian product of a Liouville domain with a disk $D(a)$. Then, for any number $0<b<a$ not in $\mathit{Spec}(Q)$, there exists a Liouville subdomain of $K\subset Q\times D(a)$ and an isomorphism:
  \begin{equation*}
    \mathfrak{K}:\mathit{SH}_{b}(K)\to \mathit{SH}_{b}(Q)
  \end{equation*}
  such that $\mathfrak{K}(\mathit{PSS}(1))=\mathit{PSS}(1)$, and such that $b\not\in \mathit{Spec}(K)$. In particular, if the $\mathit{PSS}(1)$ is zero in $\mathit{SH}_{b}(K)$, then it is also zero in $\mathit{SH}_{b}(Q)$.
\end{theorem}

Combining all these steps, we conclude that if $\Lambda$ admits no Reeb chords of length up to some number $a>c_{\mathrm{ni}}(\Omega)$, then the unit vanishes in $\mathit{SH}(T^{*}\Lambda)$. However, due to the comparison between string topology and symplectic cohomology, the unit cannot vanish in $\mathit{SH}(T^{*}\Lambda)$. This provides the desired contradiction, and proves Theorem \ref{theorem:main}, (modulo Theorems \ref{theorem:geometric-construction}, \ref{theorem:VR-main}, and \ref{theorem:filtered-kunneth}). Moreover, the argument actually gives a stronger statement; we first introduce a definition:
\begin{definition}\label{definition:uniformly-fast}
  Let $(Y,\alpha)$ be a compact cooriented contact manifold with contact form $\alpha$. If $\Lambda_{s}\subset Y$, $s\in [0,\infty)$, is a Legendrian isotopy, satisfying:
  \begin{itemize}
  \item for each smooth path $y(s)\in \Lambda_{s}$ it holds that $y^{*}\alpha\ge \d s$,
  \end{itemize}
  then we say $\Lambda_{s}$ is \emph{uniformly fast relative $\alpha$}. If $\Lambda_{s}$ is uniformly fast relative some unspecified contact form, then we simply say it is \emph{uniformly fast}.
\end{definition}
\begin{theorem}\label{theorem:strong-main} 
  Let $\Omega$ be a Liouville domain inducing a radial coordinate $r$ on $W$, and let $\alpha=\lambda/r$ be the resulting contact form on the ideal boundary $Y$ of the completion $W$. Let $\Lambda_{s}\subset Y$, $s\in [0,\infty)$, be a Legendrian isotopy which is uniformly fast relative $\alpha$, then there exists $0<a\le c_{\mathrm{ni}}(\Omega)$ such that $\Lambda_{a}\cap \Lambda_{0}\ne \emptyset$, provided $c_{\mathrm{ni}}(\Omega)<\infty$.
\end{theorem}
\begin{proof}
  This follows from the combination of Theorems \ref{theorem:geometric-construction}, \ref{theorem:VR-main}, and \ref{theorem:filtered-kunneth}.
\end{proof}
Because the Reeb flow $\Lambda_{s}=R_{s}(\Lambda_{0})$ is uniformly fast relative $\alpha$, Theorem \ref{theorem:strong-main} implies Theorem \ref{theorem:main}. Using Theorem \ref{theorem:strong-main} we prove the following in \S\ref{sec:cotangent-submersion}.
\begin{theorem}\label{theorem:cotangent-submersion}
  Suppose there is a submersion $E\to M$, where $E$ is a compact manifold, and $T^{*}E$ is $n$-invertible, e.g., $E$ could be a product of spheres. Then any uniformly fast Legendrian isotopy $\Lambda_{s}\subset ST^{*}M$, $s\in [0,\infty)$ must satisfy $\Lambda_{a}\cap \Lambda_{0}\ne \emptyset$ for some $a\in (0,\infty)$. In particular, the Arnol'd chord conjecture holds for $ST^{*}M$.
\end{theorem}

Turning now to Theorem \ref{theorem:spheres}, the method of proof is to use the framework of \cite{brocic-cant-arXiv-2025}, which allows one to prove the desired relations in the BV algebra $H_{*}(\Lambda M)$ by direct manipulation of smooth maps taking values in $M$. The proof is entirely topological and all Floer theory is isolated in our earlier work \cite{brocic-cant-arXiv-2025}.

Finally we discuss the proof of Theorem \ref{theorem:rationality} concerning the rationality constants of aspherical Lagrangians $L$ inside of $T^{*}T^{n}$. The key idea is again to use a Viterbo restriction map; however, since the Lagrangian $L$ is non-exact, Weinstein neighborhoods of $L$ are not embedded into $T^{*}T^{n}$ in an exact way, and so one cannot appeal directly to Theorem \ref{theorem:VR-main}. One instead appeals to the so-called \emph{truncated Viterbo restriction map} introduced by \cite{zhou-JSG-2022}.

\begin{theorem}\label{theorem:VR-trunc}
  Let $\iota:K\subset \Omega$ be a symplectic embedding of a Liouville domain such that the periods of the closed $1$-form $\iota^{*}\lambda_{\Omega}-\lambda_{K}$ lie in the subgroup $\rho\Z\subset \R$ for some $\rho>0$. Then, for any numbers $0<c<b$ such that $b+c<\rho$ there is a restriction map $\mathfrak{R}:\mathit{SH}_{c}(\Omega)\to \mathit{SH}_{b}(K)$ such that:
  \begin{itemize}
  \item $\mathfrak{R}\circ \Delta=\Delta\circ \mathfrak{R}$,
  \item $\mathfrak{R}(\zeta_{1}\ast \zeta_{2})=\mathfrak{R}(\zeta_{1})\ast \mathfrak{R}(\zeta_{2})$,
  \item $\mathfrak{R}\circ \mathit{PSS}=\mathit{PSS}\circ \mathfrak{R}$,
  \end{itemize}
  as in Theorem \ref{theorem:VR-main}, whenever these expressions make sense; i.e., in the second item we assume that $\zeta_{i}\in \mathit{SH}_{c_{i}}(\Omega)$ and $\mathfrak{R}(\zeta_{i})\in \mathit{SH}_{b_{i}}(K)$ for $i=1,2$ with $0<c_{i}<b_{i}$ and $c_{1}+c_{2}+b_{1}+b_{2}<\rho$.

  Finally, if $\varpi$ denotes some collection of free homotopy classes in $\Omega$, then $\mathfrak{R}$ satisfies:
  \begin{itemize}
  \item $\mathfrak{R}(\mathit{SH}_{c}(\Omega,\varpi))\subset \mathit{SH}_{b}(K,\iota^{*}\varpi\cap \kappa)$,
  \end{itemize}
  where $\kappa$ is the collection of free homotopy classes in the kernel of $\iota^{*}\lambda_{\Omega}-\lambda_{K}$.
\end{theorem}

As in \cite{zhou-JSG-2022}, given a Lagrangian $L\subset \Omega$ we apply Theorem \ref{theorem:VR-trunc} to a small Weinstein neighborhood of $L$. The next step in the argument is to apply the truncated Viterbo restriction map $\mathfrak{R}$ to classes $A_{i}^{\pm}$ as in Remark \ref{remark:on-the-constant-C} to conclude classes in $SH(T^{*}L,\kappa)$ (assuming the rationality constant of $L$ is large enough). Using the fact that $L$ is aspherical enables the use of two algebraic topology results \cite[Lemma 5.1 and 5.2]{latschev-EMS-2015}, which we use to show $\pi_{1}(L,\mathit{pt})$ admit a subgroup $\Z^{n}$ whose associated covering space is open. This leads to a contradiction, since the covering space is also an aspherical $n$-manifold, and therefore must be homotopy equivalent to $K(\Z^{n},1)\simeq T^{n}$ which has non-zero $n$th homology group (and so cannot be open). The details of this argument are given in \S\ref{sec:proof-rationality}.

\begin{remark}
  As in Theorem \ref{theorem:VR-trunc}, it is sometimes convenient to appeal to the direct sum decomposition of $\mathit{SH}_{c}(\Omega)$ into the summands generated by orbits in a given collection of free homotopy classes $\varpi$. It will be obvious from the construction that the Viterbo restriction map associated to an exact embedding $\iota:K\to \Omega$ respects this decomposition: $\mathfrak{R}(\mathit{SH}_{c}(\Omega,\varpi))\subset \mathit{SH}_{b}(K,\iota^{*}\varpi)$. We do not need to use this fact in the proof of Theorem \ref{theorem:strong-main}, but it is useful in certain cases. For instance, if there is an invertible element in $\mathit{SH}(W,\varpi)$, then every Legendrian in $\partial \Omega$ such that $\iota^{*}\varpi=\emptyset$ must have a Reeb chord.
\end{remark}

\subsection{Survey of related results}
\label{sec:survey-and-further}

The chord conjecture due to Arnol'd appears in \cite[pp. \@~15]{arnold-RMS-1986}, and it is formulated for $S^3$ with the standard contact structure. Theorem \ref{theorem:cotangent-submersion} solves the Arnol'd chord conjecture for $\Lambda\subset ST^{*}M$ when it applies.

\subsubsection{Higher dilations}
\label{sec:higher-dilations}

If one allows constraints on the topology of $\Lambda$ relative the topology of $M$, then one can deduce quite a bit by combining the following two propositions:
\begin{proposition}[{\cite[Corollary 1.1.6]{zhao-thesis-2016}} using \cite{goodwillie-topology-1985}]\label{claim:zhao}
  A compact orientable manifold $M$ is rationally essential\footnote{Every connected manifold has a classifying map $M\to BG$ where $G=\pi_{1}(M,\mathit{pt})$; if the induced map in rational homology kills $[M]$ (we assume $M$ is orientable so $[M]$ is well-defined), then we say $M$ is rationally inessential; otherwise $M$ is said to be rationally essential. Rationally inessential manifolds contain all simply connected manifolds, and form an ideal with respect to Cartesian product (in the space of orientable manifolds).} if and only if the unit element is non-zero in the $S^{1}$-equivariant symplectic homology $\mathit{SH}^{S^{1}}(M,\Q,\mathfrak{L})$, defined with rational coefficients and twisted by the local coefficient system $\mathfrak{L}$ obtained by transgressing the second Stiefel Whitney class of $TM$.\hfill$\square$
\end{proposition}

\begin{remark}
  The vanishing of the unit is often expressed as the existence of a ``higher dilation.''  Dilation classes were introduced in \cite{seidel-solomon-GAFA-2012,seidel-inventiones-2014} and further developed from the framework of $S^{1}$-equivariant homology by \cite{zhao-thesis-2016,zhou-jtopol-2021,zhou-adv-math-2022,zhou-JSG-2022,li-yin-arXiv-2023,li-QT-2024}.

  We note that \cite{albers-cieliebak-frauenfelder-PLMS-2016,frauenfelder-pajitnov-JFPTA-2017} also appeal to the result of \cite{goodwillie-topology-1985}. In \cite{frauenfelder-pajitnov-JFPTA-2017} it is shown that the $\pi_{1}$-sensitive Hofer--Zehnder capacity is finite in the cotangent bundles of all compact rationally inessential manifolds.
\end{remark}

\begin{proposition}[Truncated Viterbo Restriction {\cite{zhou-JSG-2022}}]\label{claim:zhou}
  If $\Lambda\subset \partial \Omega$ is a Legendrian without Reeb chords, there is a restriction morphism:
  $$\mathit{SH}^{S^{1}}(\Omega,\mathbf{k},\mathfrak{L})\to
  \mathit{SH}^{S^{1}}(T^{*}(\Lambda\times S^{1}),\mathbf{k},\mathfrak{L}'),$$ which sends the unit element to the unit element; here $\mathbf{k}$ is a coefficient field. The local system $\mathfrak{L}'$ is obtained by restriction the local system $\mathfrak{L}$ to a Weinstein neighborhood of a Lagrangian embedding $\Lambda\times S^{1}\to \Omega$.\hfill$\square$
\end{proposition}

By appealing to
\cite{zhao-thesis-2016}, \cite{zhou-JSG-2022}
concludes:
\begin{claim}[{\cite{zhou-JSG-2022}}]\label{claim:zhou-rationally-inessential}
  If $\Lambda$ is an aspherical
  Legendrian in the unit cotangent bundle of
  a compact orientable rationally inessential
  manifold $M$, then $\Lambda$ has a Reeb chord
  for every choice of contact
  form.\hfill$\square$
\end{claim}
\begin{proof}[Summary of proof]
  Using Proposition \ref{claim:zhou} and Claim \ref{claim:zhao}, one obtains:
  \begin{equation*}
    \text{unit vanishes in }\mathit{SH}^{S^{1}}(T^{*}(\Lambda\times S^{1}),\Q,\mathfrak{L}'),
  \end{equation*}
  where $\mathfrak{L}'$ is obtained by transgressing some class in $H^{2}(\Lambda\times S^{1},\Z/2\Z)$. This is impossible if $\Lambda\times S^{1}$ is aspherical, as we now explain.

  Abbreviate $N=\Lambda\times S^{1}$. The first step is to appeal to \cite{abouzaid-EMS-2015} to conclude there is a commutative diagram:
  \begin{equation*}
    \begin{tikzcd}
      H_{*}(N,\mathfrak{o}\otimes \Q)\arrow[d,"\simeq"]\arrow[r,"\iota"]&H_{*}(\Lambda N,\Q,\mathfrak{o}\otimes \mathfrak{L}'\otimes \mathfrak{L})\arrow[d,"\simeq"]\\
      H^{*}(T^{*}N,\Q)\arrow[r,"PSS"]&SH(T^{*}N,\Q,\mathfrak{L}').
    \end{tikzcd}
  \end{equation*}
  The local system $\mathfrak{L}$ is also obtained by transgressing a class in $H^{2}(N,\Z/2\Z)$, while $\mathfrak{o}$ is the orientation local system for the base manifold. The top arrow $\iota$ is an isomorphism onto the contractible summand, since $N$ is aspherical. It follows that the PSS map is also an isomorphism onto the contractible summand.

  This implies that the equivariant PSS map:
  \begin{equation*}
    H^{*}(T^{*}N,\Q)\otimes \Q[u^{-1}]\to SH^{S^{1}}(T^{*}N,\Q,\mathfrak{L}') 
  \end{equation*}
  cannot kill the unit, by a spectral sequence argument applied to the $u$-adic filtration inherent to $S^{1}$-equivariant homology; see \cite{zhao-thesis-2016,zhou-adv-math-2022,li-QT-2024}. We recall the argument for the reader's convenience. The contractible summand of $SH^{S^{1}}(T^{*}N,\Q,\mathfrak{L}')$ is the homology of a complex: $$(C_{0}\oplus C_{+})\otimes \Q[u^{-1}]$$ with deformed differential:
  \begin{equation*}
    d_{S^{1}}=d+u\Delta+u^{2}\delta_{2}+u^{3}\delta_{3}+\dots
  \end{equation*}
  where $\delta_{i}$ can be considered as ``higher order'' BV operators. One sets up the chain complex so that:
  \begin{itemize}
  \item $(C_{0},d_{0})$ is a subcomplex whose homology computes $H^{*}(T^{*}N,\Q)$,
  \item $H(C_{0}\oplus C_{+},d_{0})$ is the contractible summand of $SH(T^{*}N,\Q,\mathfrak{L}')$,
  \item the inclusion of $C_{0}$ into $C_{0}\oplus C_{+}$ induces the map $\mathit{PSS}$ on homology,
  \item all higher order terms $\Delta,\delta_{2},\delta_{3},\dots$ vanish on $C_{0}\otimes \Q[u^{-1}]$.
  \end{itemize}
  Suppose we have a solution of:
  \begin{equation*}
    d_{S^{1}}(u^{-k}\zeta_{k}+\dots+u^{-1}\zeta_{1}+\zeta_{0})=1\in C_{0}.
  \end{equation*}
  If $k>0$ then $d_{0}\zeta_{k}=0$. Since $C_{0}\to C_{0}\oplus C_{+}$ is an isomorphism on $d_{0}$ homology, we can assume that $\zeta_{k}=\eta_{k}+d_{0}\mu_{k}$ for $\eta_{k}\in C_{0}$. Then, adding $0=d_{S^{1}}(d_{S^{1}}(u^{-k}\mu_{k}))$ to both sides yields:
  \begin{equation*}
    d_{S^{1}}(u^{-k}\eta_{k}+u^{-k+1}(\zeta_{k-1}+\Delta\mu_{k})+\dots+u^{-1}(\zeta_{1}+\delta_{k-1} \mu_{k})+\zeta_{0}+\delta_{k}\mu_{k})=1
  \end{equation*}
  Now, since all the higher terms in $d_{S^{1}}$ vanish on $\eta_{k}\in C_{0}$, we can simply remove it from the equation and conclude:
  \begin{equation*}
    d_{S^{1}}(u^{-k+1}(\zeta_{k-1}+\Delta\mu_{k})+\dots+u^{-1}(\zeta_{1}+\delta_{k-1} \mu_{k})+\zeta_{0}+\delta_{k}\mu_{k})=1.
  \end{equation*}
  By repeating this argument we conclude:
  \begin{equation*}
    d_{0}(\zeta_{0}+\delta_{k}\mu_{k}+\delta_{k-1}\mu_{k-1}'+\dots)=1,
  \end{equation*}
  contradicting the fact that $1$ is not an exact cycle with respect to $d_{0}$.
\end{proof}

\subsubsection{Orientation tricks}
\label{sec:orientation-tricks}

Let us note one curiously
general example:
\begin{claim}\label{claim:spin-CPd}
  Let $M$ be a compact manifold such that $w_{2}(M)$ evaluates non-trivially on some element in $\pi_{2}(M)$. For example, $M=N\times \C P^{2n}$ for any positive integer $n$ and any compact smooth manifold $N$. Then every spin Legendrian in $ST^{*}M,$ admits a Reeb chord for all contact forms.
\end{claim}
\begin{proof}
  By \cite{albers-frauenfelder-oancea-mathann-2017} (building on \cite{seidel-note-2010} in the case $M=\C P^{2}$), it holds that $\mathit{SH}(T^{*}M,\Q)=0$ and so the same argument used in Claim \ref{claim:zhou} can be applied.
\end{proof}
\begin{remark}
  It seems to be a somewhat interesting
  question how the ``spin'' assumption in
  Claim \ref{claim:spin-CPd} can be lifted
  (without imposing further conditions on the embedding of $\Lambda$ into $ST^{*}M$). One can obtain more consequences from \cite{albers-frauenfelder-oancea-mathann-2017} if one specifies how the inclusion $\Lambda\to ST^{*}M$ acts on $H^{2}$.
\end{remark}

\subsubsection{Aspherical ambient spaces}
\label{sec:asph-ambi-spac}

Let us comment that Theorem
\ref{theorem:cotangent-submersion} applies to
$ST^{*}T^{n}$, because $T^{n}$ is a Cartesian
product of $n$ copies of $S^{1}$. This is
notable because $T^{n}$ is aspherical and
hence the symplectic cohomology does not
admit any type of higher dilation classes, or
any trickery with orientations as in
Claim \ref{claim:spin-CPd}. In this respect,
i.e., applicability to aspherical manifolds,
our result is closer to the work of
\cite{irie-JEMS-2014} which uses the product
structure to bound the Hofer Zehnder capacity
of cotangent bundles of manifolds admitting
circle actions with non-contractible orbit
class.

\subsubsection{Fillable Legendrians}
\label{sec:fillable-legendrians}

One can also restrict the class of Legendrians in another ways: by requiring that they admit Lagrangian fillings; see, e.g., \cite{ritter-jtopol-2013}. If the fillings are assumed to be exact, then one can use wrapped Floer cohomology; see, e.g., \cite{abouzaid-seidel-GT-2010}. The Arnol'd chord conjecture is established for conormal Legendrians via a comparison with the Morse homology of path spaces, see \cite{brocic-cant-shelukhin-math-ann-2025} which builds on work of \cite{abbondandolo-schwarz-CPAM-2006, abbondandolo-portaluri-schwarz-JFPTA-2008}.

\subsubsection{Strong Arnol'd chord conjecture}
\label{sec:strong-arnold-chord}

An interesting research direction is the ``strong'' version of the conjecture, as studied in \cite{kang-jtopol-2026,kang-zhang-arXiv-2024,cant-arXiv-2026}, which asks whether there is always a chord with distinct endpoints. It seems difficult to establish the strong Arnol'd conjecture from our methods, since the upper bound we obtain on the length of the Reeb chord is always at least the systole of $\lambda|_{\bd \Omega}$ (so we cannot a priori guarantee chords shorter than all orbits, which is the surefire way to conclude the strong Arnol'd conjecture).

\subsubsection{Time-dependent Reeb flows}
\label{sec:time-dependent-reeb}

The method is robust, in that it allows one to prove the existence of chords for time-dependent Reeb flows, as long as they are uniformly fast; see Definition \ref{definition:uniformly-fast} and Theorem \ref{theorem:strong-main}.  We note that other methods of producing Reeb chords (such as bubbling analysis for almost complex structures of SFT type) do not seem to easily produce chords for uniformly fast time-dependent flows.

\subsubsection{Cotangent bundles of manifolds with negative sectional curvature}
\label{sec:cotang-bundl-manif}

The restrictions from this paper are related to interactions between product structures, BV operators, and the PSS morphisms. When $\Omega\subset T^{*}M$, and $M$ has negative sectional curvature, there are seemingly not enough relations between the top degree PSS classes and the other classes to apply our methods. It is natural to wonder if there a symplectic cohomology based upper bound on the length of Reeb chords of Legendrian knots in $\bd\Omega$ for domains $\Omega\subset T^{*}\Sigma_{g}$, where $\Sigma_{g}$ is a surface of genus $g>1$. Note that, since $ST^{*}\Sigma_{g}$ is a contact 3-manifold, the existence of Reeb chords is known \cite{hutchings-taubes-MRL-2011,hutchings-taubes-GT-2013}. One can also compare with the problem on the finiteness of the Hofer-Zehnder capacity of compact subsets in $T^{*}\Sigma_{g}$, \cite{bimmermann-arch-math-2024,bimmermann-et-al-arXiv-2026}.

\subsection{Acknowledgements}
\label{sec:acknowledgements}

The authors would like to thank E.\@~Shelukhin for many useful discussions surrounding these topics. The key ideas for this project were developed during the second author's visit to Augsburg in May 2025; the authors would like to thank Universit\"at Augsburg for providing a pleasant work environment. The second author thanks P.\@~Biran at ETHZ, J.\@~Shang at LMO, J.\@~Zhang at USTC IGP, W.\@~Gong at BNU, and P.\@ Salomão at SUSTech SICM for valuable discussions during the preparation of this work. The authors also wish to warmly thank Z.\@~Zhou and K.\@~Guo for communicating their related work \cite{guo-zhou-arXiv-2026} (developed contemporaneously with work the present paper) during the May 2026 program at the Tianyuan Mathematics Research Center in Kunming.

The first author was funded by the Deutsche Forschungsgemeinschaft (DFG, German Research Foundation) – 541525489. The second author was funded by the ANR project CoSy and an NSERC/CRSNG Discovery grant.

\section{Proofs}
\label{sec:proofs}

The contents of this section are:
\begin{itemize}
\item \S\ref{sec:from-legendr-stab} on the existence of embeddings of stabilized codisk bundles;
\item \S\ref{sec:kunneth-morphism} on the Künneth morphism;
\item \S\ref{sec:string-topology-products-spheres} on the string topology of products of spheres;
\item \S\ref{sec:proof-rationality} on rationality constant of Lagrangians in $T^{*}T^{n}$.
\end{itemize}

Due to the more-or-less standard nature of the Viterbo restriction map, see \cite{viterbo-GAFA-1999,mclean-GT-2009,abouzaid-seidel-GT-2010,ritter-jtopol-2013}, and the fact the truncated version already appears in \cite{zhou-JSG-2022}, we skip directly from \S\ref{sec:from-legendr-stab} to \S\ref{sec:kunneth-morphism} and confine the proof of Theorems \ref{theorem:VR-main} and \ref{theorem:VR-trunc} to Appendix \ref{sec:viterbo-restriction}.

\subsection{From Legendrians to embeddings of stabilized codisk bundles}
\label{sec:from-legendr-stab}

The goal of this section is to prove Theorem \ref{theorem:geometric-construction}. The first part of the construction involves a coisotropic embedding:
\begin{lemma}\label{lemma:coiso-construct}
  Let $\Omega,\alpha,\Lambda_{t},a$ be as in the statement of Theorem \ref{theorem:geometric-construction}. There exists an embedding $\varphi:\Lambda\times D(a)\to \Omega,$ taking values in the interior, such that:
  \begin{equation}\label{eq:pullback-formula}
    \varphi^{*}\lambda=\lambda_{D(a)}
  \end{equation}
  where $\lambda_{D(a)}$ is (the pullback of) a Liouville form on $D(a)$ (defining its standard symplectic structure). In particular, $\varphi$ is a coisotropic embedding.
\end{lemma}
\begin{proof}
  First let $\psi:\Lambda\times [0,1]\to \partial \Omega$ be some parametrization of the Legendrian isotopy $\Lambda_{t}$, so that $\psi(\Lambda\times \set{t})=\Lambda_{t}$. Then, by the hypotheses of Theorem \ref{theorem:geometric-construction}, it holds that $\psi$ is an embedding and:
  \begin{equation*}
    \psi^{*}\lambda=h\d t
  \end{equation*}
  where $h:\Lambda\times [0,1]\to (a,\infty)$ is a smooth function. By compactness of $\Lambda\times [0,1]$, we may suppose that $h>a'>a$ for some $a'$.

  Now consider the modification by a domain dependent Liouville flow:
  \begin{equation*}
    \tilde{\psi}(x,t)=Z_{\log(a'/h(x,t))}(\varphi(x,t)).
  \end{equation*}
  This is still an embedding (since $\varphi$ remains an embedding even after quotienting by the Liouville flow), and it takes values in the interior of $\Omega$. An easy computation yields:
  \begin{equation*}
    \tilde{\psi}^{*}\lambda=a'\d t.
  \end{equation*}
  The next step is to extend the embedding as:
  \begin{equation*}
    \phi(x,s,t)=Z_{s}\tilde{\psi}(x,t)\text{ where }\phi:\Lambda\times (-\infty,0]\times [0,1]\to \Omega.
  \end{equation*}
  Another standard computation yields:
  \begin{equation*}
    \phi^{*}\lambda=a'e^{s}\d t.
  \end{equation*}
  Now we observe that the area of $(-\infty,0]\times [0,1]$ with respect to the area form $a'e^{s}\d t$ is $a'$, and so, since $a'>a$, we can find a smooth area preserving embedding $w:D(a)\to (-\infty,0]\times [0,1]$ such that $w^{*}(a'e^{s}\d t)=\lambda_{D(a)}$, where $\lambda_{D(a)}$ is some Liouville form on $D(a)$ for its standard symplectic structure.

  Finally we define $\varphi:\Lambda\times D(a)\to \Omega$ by the formula:
  \begin{equation*}
    \varphi(x,z)=\phi(x,w(z)).
  \end{equation*}
  This map $\varphi$ satisfies the first part of the lemma. It remains to explain why $\varphi$ is a coisotropic embedding. It is clear that $\d\varphi(T\Lambda)$ lies in the $\omega$-orthogonal complement of $\varphi(T\Lambda\times D(a))$, because of the pullback formula \eqref{eq:pullback-formula}. By dimensional reasons, it follows that $\d\varphi(T\Lambda)$ must equal the entire $\omega$-orthogonal complement. Thus the embedding is coisotropic.
\end{proof}

The next step is to appeal to the following general fact:
\begin{lemma}\label{lemma:gotay}
  Let $N$ be a compact smooth manifold with boundary, and suppose there are two coisotropic embeddings:
  \begin{equation*}
    j_{1}:N\to (M_{1},\omega_{1})\text{ and }j_{2}:N\to (M_{2},\omega_{2})
  \end{equation*}
  such that $j_{1}^{*}\omega=j_{2}^{*}\omega$. Then a neighborhood of $j_{1}(N)\subset M$ is symplectomorphic to a neighborhood of $j_{2}(N)\subset M_{2}$, and the symplectomorphism can be chosen to send $j_{1}(x)$ to $j_{2}(x)$ for $x\in N$.
\end{lemma}
\begin{proof}
  See \cite{gotay-PAMS-1982} and \cite[Exercise 3.4.18]{mcduff-salamon-book-2017}.
\end{proof}
We apply this fact with:
\begin{equation*}
  N=\Lambda\times D(a)\quad M_{1}=T^{*}\Lambda\times \C\quad M_{2}=\Omega.
\end{equation*}
In particular, we conclude from Lemma \ref{lemma:gotay} a symplectic embedding:
\begin{equation*}
  \Phi:D_{\epsilon}T^{*}\Lambda\times D(a)\to \Omega
\end{equation*}
which agrees with the map $\varphi$ from Lemma \ref{lemma:coiso-construct} on $\Lambda\times D(a)$. As it is a symplectic embedding, it follows that:
\begin{equation*}
  \Phi^{*}\lambda=p\d q+\lambda_{D(a)}+\beta
\end{equation*}
where $\beta$ is a closed 1-form (here $p\d q$ is the standard Liouville form on $T^{*}\Lambda$). Since $\Phi$ agrees with $\varphi$ on the zero section, and $D_{\epsilon}T^{*}\Lambda\times D(a)$ deformation retracts onto $\Lambda\times D(a)$ we conclude that all periods of $\beta$ vanish, and hence $\beta$ is an exact 1-form. This completes the proof of Theorem \ref{theorem:geometric-construction}. \hfill$\square$

\subsection{Künneth morphism}
\label{sec:kunneth-morphism}

A Liouville manifold $(W,\omega,Z)$ involves the data of a Liouville vector field $Z$ such that $L_{Z}\omega=\omega$. Let us denote by $Z_{s}$ the time-$s$ flow. Hamiltonian Floer theory of $(W,\omega,Z)$ is well-established provided one uses Hamiltonian systems and $\omega$-tame almost complex structures that are invariant under $Z_{s}$, $s\ge 0$, outside of a compact set; see, e.g., \cite{brocic-cant-JFPTA-2024}. It is in this context that one defines the persistence module $\mathit{SH}_{c}(\Omega)$ used throughout this paper.

The product $W\times \C$ is another Liouville manifold with diagonal Liouville vector field:
\begin{equation*}
 Z_{W\times \C}=Z_{W}+Z_{\C} 
\end{equation*}
where $Z_{\C}$ is the radial Liouville vector field on $\C$, with its standard symplectic form.

The naive idea of the Künneth morphism is that, if $H_{t}$ generates a non-degenerate Hamiltonian
system on $W$ (with positive slope $b$ as in the statement of Theorem \ref{theorem:filtered-kunneth}), then:
\begin{equation*}
  H_{t}+\tau \pi\abs{z}^{2}\text{ where }\tau=b/a
\end{equation*}
generates a non-degenerate Hamiltonian system on $W\times \C$, and the vector space underlying the Floer complex $\mathit{CF}(H_{t}+\tau\pi\abs{z}^{2})$ is identified with the Floer complexes $\mathit{CF}(H_{t})$ computed in $W$. If one picks a split almost complex structure $J_{W\times \C}=J_{W}\oplus J_{\C}$, then the differential respects this identification, and one obtains an isomorphism of complexes:
\begin{equation}\label{eq:isomorphism-of-complexes}
  \mathit{CF}(H_{t}+\tau \pi\abs{z}^{2},J_{W\times \C})\to \mathit{CF}(H_{t},J_{W}).
\end{equation}
Since $\tau\in (0,1)$, this chain level identification also respects the unit element (this is no longer true when $\tau\not\in (0,1)$).

The problem in using this observation in symplectic cohomology theory is that $H_{t}+\pi \abs{z}^{2}$ and $J_{W}+J_{\C}$ are not $Z$-invariant outside of a compact set, and so our preferred framework (as used in \cite{brocic-cant-JFPTA-2024,brocic-cant-arXiv-2025} and throughout the rest of this paper) cannot be applied. Existing literature \cite{oancea-JSG-2008,mclean-GT-2009} successfully establishes a ``Künneth isomorphism'' between the colimits:
\begin{equation*}
  \mathit{SH}(W\times \C)\simeq \mathit{SH}(W)\otimes \mathit{SH}(\C).
\end{equation*}
Such a statement is not ideal for our purposes since $\mathit{SH}(\C)=0$; we instead require a ``filtered'' Künneth map $\mathfrak{K}$ as in Theorem \ref{theorem:filtered-kunneth}.


\subsubsection{Choice of almost complex structure}
\label{sec:choice-almost-complex}

As in Theorem \ref{theorem:filtered-kunneth}, fix:
\begin{itemize}
\item a Liouville domain $Q$, with radial coordinate $r_{1}$, 
\item an $\omega$-tame almost complex structure $J_{1}$ on the completion $W$, which is invariant under $Z_{s}$, $s\ge 0$, outside of $Q=\set{r_{1}\le 1}$,
\item the standard complex structure $J_{2}$ on $\C$,
\item the radial coordinate $r_{2}=a^{-1}\pi\abs{z}^{2}$ for the Liouville domain $D(a)$.
\end{itemize}

The complex structure:
\begin{equation*}
  Z_{s}^{*}J_{1}=\d Z_{s}^{-1}J_{1}\d Z_{s}
\end{equation*}
on $W$ can be considered as ``neck stretched,'' loosely speaking.

Let $\ell(r)$ be a smooth approximation of $\max\set{\log(r),0}$, such that:
\begin{itemize}
\item $\ell(r)=\log(r)$ when $r\ge 1$,
\item $\ell(r)\le 0$ for $r\le 1$,
\item $\ell(r)$ is constant on the interval $r\in [0,\frac{1}{2}]$.
\end{itemize}
See \S\ref{sec:ham-systems-for-kunneth} for an explicit choice of such $\ell$.

Define an almost complex structure on $T(W\times \C)$:
\begin{equation}\label{eq:special-J}
  J=Z_{-\ell(r_{2})}^{*}(J_{1})\oplus J_{2}.
\end{equation}
Then we have:
\begin{lemma}\label{lemma:LE}
  The almost complex structure $J$ is
  Liouville equivariant outside of the
  rectangle
  $\set{r_{1}\le 1}\cap \set{r_{2}\le
    1}$. That is to say, if
  $(w_{1},w_{2})$ lies outside of this
  rectangle, and $s>0$, it holds that
  $(Z_{s}^{*}J)_{(w_{1},w_{2})}=J_{(w_{1},w_{2})}$.  Moreover,
  $J$ is genuinely split:
  $$J_{(w_{1},w_{2})}=(Z^{*}_{-\ell(0)}J_{1})_{w_{1}}\oplus
  (J_{2})_{w_{2}}$$ on the strip
  $\set{r_{2}\le 1/2}$.
\end{lemma}
\begin{proof}
  Let $r_{i}=r_{i}(w_{i})$. If $r_{2}$ is larger than $1$, we compute:
  \begin{equation*}
    \begin{aligned}
      (Z^{*}_{s}J)_{(w_{1},w_{2})}
      &=(Z^{*}_{s-\log(e^{s}r_{2})}J_{1})_{w_{1}}\oplus J_{2,w_{2}}\\
      &=(Z^{*}_{-\log(r_{2})}J_{1})_{w_{1}}\oplus J_{2,w_{2}},\\
      &=(Z^{*}_{-\ell(r_{2})}J_{1})_{w_{1}}\oplus J_{2,w_{2}}=(J)_{(w_{1},w_{2})}.
    \end{aligned}
  \end{equation*}
  On the other hand if $r_{1}$ is larger
  than $1$ while $r_{2}$ is smaller than
  $1$, we compute:
  \begin{equation*}
    \begin{aligned}
      (Z^{*}_{s}J)_{(w_{1},w_{2})}
      &=(Z^{*}_{s-\ell(e^{s}r_{2})}J_{1})_{w_{1}}\oplus J_{2,w_{2}}\\
      &=J_{1,w_{1}}\oplus J_{2,w_{2}}\\
      &=(Z^{*}_{-\ell(r_{2})}J_{1})_{w_{1}}\oplus J_{2,w_{2}},
    \end{aligned}
  \end{equation*}
  where we use the fact that $J_{1}$ is
  equivariant outside of $\set{r_{1}\ge 1}$,
  and:
  \begin{itemize}
  \item $s-\ell(e^{s}r_{2})\ge 0$, in
    the second equality,
  \item $-\ell(r_{2})\ge 0$, in the
    third inequality.
  \end{itemize}
  Finally, if $r_{2}$ is less than $1/2$, the splitting of $J$ holds from the requirement that $\ell(r_{2})=\ell(0)$ is constant in this region.
\end{proof}

\subsubsection{Hamiltonian systems for the Künneth morphism}
\label{sec:ham-systems-for-kunneth}

We continue with the set-up of \S\ref{sec:choice-almost-complex}. The construction of the Künneth map is defined in a similar way to the ``naive'' Künneth map described in \S\ref{sec:kunneth-morphism}; see equation \eqref{eq:isomorphism-of-complexes}. As in the statement of Theorem \ref{theorem:filtered-kunneth}, we will find a domain $K\subset Q\times D(a)$ and directly identify a chain complex computing $\mathit{SH}_{b}(K)$ with a chain complex computing $\mathit{SH}_{b}(Q)$; this identification produces the desired map $\mathfrak{K}$.

The primary goals of \S\ref{sec:kunneth-morphism} are to construct this map and prove it respects $\mathit{PSS}(1)$; a secondary goal is to make the map $\mathfrak{K}$ canonical. For the sake of canonicity, we introduce some functions via explicit formulas:
\begin{definition}\label{definition:cut-off}
  Define a function $\beta:\R\to [0,1]$ by the formula:
  \begin{equation*}
    \beta(s):=\left\{
      \begin{aligned}
        &0&&\text{ for }s\le 0\\
        &\textstyle m^{-1}\int_{0}^{s}e^{-1/\sigma^{2}}e^{-1/(\sigma-1)^{2}}d\sigma&&\text{ for }s\in [0,1]\\
        &1&&\text{ for }s\ge 1
      \end{aligned}
    \right.
  \end{equation*}
  where $m=\int_{0}^{1}e^{-1/\sigma^{2}}e^{-1/(\sigma-1)^{2}}d\sigma$. This function is a standard smooth cut-off function which increases monotonically from $0$ to $1$ on the interval $[0,1]$. Also define:
  \begin{itemize}
  \item $f(r):=\int_{0}^{r}\beta(2s-1)\d s+1-\int_{0}^{1}\beta(2s-1)\d s$,
  \end{itemize}
  so that $f$ is convex, non-decreasing, constant on $[0,\frac{1}{2}]$, and satisfies $f(r)=r$ for $r\ge 1$. Finally define:
  \begin{itemize}
  \item $\ell(r):=\log(f(r))$,
  \end{itemize}
  which makes the choice of $\ell(r)$ in \S\ref{sec:choice-almost-complex} canonical.
\end{definition}

The map $\mathfrak{K}$ is defined via the following family of Hamiltonian systems. 
\begin{definition}\label{definition:admissible-for-kunneth}
  Let $C^{\infty}(Q\times \R/\Z)$ be the topological space of compactly supported, time dependent, smooth functions $P_{t}$ on $Q$, and let $\mathscr{U}$ be an open neighborhood of $0$ in $C^{\infty}(Q\times \R/\Z)$.
  
  Let us say that a time-dependent Hamiltonian function $H_{t}$ on $W\times \C$ is admissible for the Künneth map for $Q\times D(a)$ and slope $b<a$, relative $\mathscr{U}$ and a number $\delta>0$, provided that:
  \begin{equation}\label{eq:crazy-formula}
    H_{t}=b e^{\ell(r_{2})}\delta f(\delta^{-1}e^{-\ell(r_{2})}r_{1})+br_{2}+\beta(2-2r_{2})P_{t}
  \end{equation}
  where $P_{t}\in \mathscr{U}$ and:
  \begin{enumerate}[label=(K\arabic*)]
  \item\label{item:K1} all orbits of $H_{t}$ lie in the submanifold $\set{r_{2}=0}=W\times \set{0}$, and,
  \item\label{item:K2} the moduli space of Floer cylinders for $(H_{t},J)$ is cut transversally, where $J$ is given by \eqref{eq:special-J}.
  \end{enumerate}
  We denote by $\mathscr{O}_{b}(\delta,\mathscr{U})$ the set of such systems $H_{t}$ satisfying \ref{item:K1} and \ref{item:K2} with $P_{t}\in \mathscr{U}$.
\end{definition}
\begin{definition}\label{definition:Kdelta}
  In the context of Definition \ref{definition:admissible-for-kunneth}, define:
  \begin{equation*}
    K_{\delta}:=\set{\delta e^{\ell(r_{2})}f(\delta^{-1}e^{-\ell(r_{2})}r_{1})+r_{2}\le 1}\subset Q\times D(a).
  \end{equation*}
  Then $K_{\delta}$ is a smooth Liouville domain\footnote{To check $K_{\delta}\subset Q\times D(a)$, observe that:
  \begin{equation*}
    \max\set{r_{1},r_{2}}\le cf(c^{-1}r_{1})+r_{2}\text{ for all }c>0
  \end{equation*}
  since $f$ is convex and $f(x)=x$ for $x\ge 1$.} defining the Liouville structure on $W\times \C$ such that $H_{t}=r_{K_{\delta}}$ holds outside of a compact set (here $r_{K_{\delta}}$ is the radial coordinate for the domain $K_{\delta}$) when $H_{t}$ is defined by \eqref{eq:crazy-formula}. 
\end{definition}
\begin{remark}\label{remark:significant-canonical}
  To simplify the exposition, we will not discuss changes of almost complex structure and will always use $J$ given by \eqref{eq:special-J} for a fixed $J_{1}$. This slight lack of canonicity is easily remedied by including $J_{1}$ as an auxiliary choice, and this strategy works without any problems (except for making the notation more cumbersome). We adopt a middle ground and try to make the construction canonical with regard to the ``significant'' choice of $P_{t}$.
\end{remark}

\begin{lemma}
  For any slope $b<a$ and any open set $\mathscr{U}$ and sufficiently small number $\delta$ as in Definition \ref{definition:admissible-for-kunneth} the set $\mathscr{O}_{b}(\delta,\mathscr{U})$ is non-empty.
\end{lemma}
\begin{proof}
  For \ref{item:K1}, we compute the Hamiltonian vector field $X_{H_{t}}$ with respect to the splitting $TW\oplus T\C$:
  \begin{equation}\label{eq:XHT-splitting}
    X_{H_{t}}=\left[
      \begin{matrix}
        bf'(\delta^{-1}e^{-\ell(r_{2})}r_{1})X_{r_{1}}+\beta(2-2r_{2})X_{P_{t}}\\
        (bS-2\beta'(2-2r_{2})P_{t})X_{r_{2}}
      \end{matrix}
    \right]
  \end{equation}
  where:
  \begin{itemize}
  \item $S=\ell'(r_{2})e^{\ell(r_{2})}\delta [f(\rho)-\rho f'(\rho)]+1$ abbreviating $\rho=\delta^{-1} e^{-\ell(r_{2})}r_{1}$.
  \end{itemize}
  In particular, by picking $\delta$ and $P_{t}$ small enough, we may ensure that:
  \begin{equation*}
    0<bS-2\beta'(2-2r_{2})P_{t}< a,
  \end{equation*}  
  since $S$ can be made as close to $1$ as desired by making $\delta$ small (and $P_{t}$ can also be made vanishingly small). Since the minimal period of $X_{r_{2}}$ is $a$ (recall that $r_{2}=a^{-1}\pi \abs{z}^{2}$ is the radial coordinate for $D(a)$) it follows that any initial point lying above the region $r_{2}\ne 0$ cannot lie on a one-periodic orbit of $X_{H_{t}}$. Thus \ref{item:K1} holds.

  Part \ref{item:K2} follows easily from \ref{item:K1} and the formula \eqref{eq:XHT-splitting} --- one can simply use the generic perturbation term $P_{t}$ to render all orbits non-degenerate and all Floer cylinders regular (as all Floer cylinders must pass through the region where $P_{t}$ is active). Here it is important that $b\not\in \mathit{Spec}(Q)$ so that all orbits actually lie in $Q$ (where $P$ is active).
\end{proof}

\subsubsection{Floer cylinders for the Künneth morphism}
\label{sec:floer-cylind-kunn}

In this section we prove structural results for the Floer complex $\mathit{CF}(H_{t},J)$ provided $H_{t}\in \mathscr{O}_{b}(\delta,\mathscr{U})$ and $\delta$ and the open set $\mathscr{U}$ are sufficiently small.

\begin{lemma}\label{lemma:floer-cylinders-KMAP}
  For sufficiently small $\delta$ and $\mathscr{U}$ as in Definition \ref{definition:admissible-for-kunneth} all Floer cylinders for $(H_{t},J)$ where $H_{t}\in \mathscr{O}_{b}(\delta,\mathscr{U})$ lie in the submanifold $W\times \set{0}$.

  Moreover, every Floer differential cylinder for $(H_{t}|_{W\times \set{0}},J|_{W\times \set{0}})$ remains a valid Floer differential cylinder for $H_{t},J$.
\end{lemma}
\begin{proof}
  By \ref{item:K1}, we already know that all orbits of $H_{t}\in \mathscr{O}_{b}(\delta,\mathscr{U})$ lie in the submanifold $W\times \set{0}$. To prove that the connecting trajectories for $(H_{t},J)$ also lie in this submanifold, we appeal to a compactness argument.

    The key is the formula \eqref{eq:XHT-splitting} and the fact that $\ell(r_{2})$ is constant on the region $r_{2}\le 1/2$. Indeed, if $r_{2}\le 1/2$, then \eqref{eq:XHT-splitting} simplifies to the split system:
  \begin{equation*}
    X_{H_{t}}=\left[
      \begin{matrix}
        bf'(\delta^{-1}e^{-\ell(0)}r_{1})X_{r_{1}}+X_{P_{t}}\\
        bX_{r_{2}}
      \end{matrix}
    \right].
  \end{equation*}
  Since $J$ is also split in the region $r_{2}\le 1/2$ (Lemma \ref{lemma:LE}), we conclude that any Floer cylinder $u$ contained in the region $\set{r_{2}\le 1/2}$ equals $u=(u_{1},u_{2})$ where $u_{2}$ is a Floer cylinder for $(bX_{r_{2}},J_{2})$. Since $b\in (0,a)$, the only Floer cylinders for $(b X_{r_{2}},J_{2})$ are the constant ones located at the origin. Thus, any Floer cylinder for $(H_{t},J)$ not contained in $W\times \set{0}$ must intersect the region $\set{r_{2}>1/2}$. This conclusion is stable under taking limits: if we consider a sequence $\delta_{n}\to 0$ and $\mathscr{U}_{n}$, so that $P_{n,t}\in \mathscr{U}_{n}$ implies $P_{n,t}\to 0$, then any sequence of Floer cylinders $u_{n}=(u_{1,n},u_{2,n})$ intersecting $\set{r_{2}>1/2}$ will have a subsequence so that $u_{2,n}$ converges (after reparametrization by a translation) to a Floer cylinder for $b X_{r_{2}}$. This is because since the coefficient:
  \begin{equation*}
    (bS-2\beta'(2-2r_{2})P_{t})\quad\text{from \eqref{eq:XHT-splitting}}
  \end{equation*}
  converges uniformly to $b$ as $n\to\infty$. This completes the first part of the proof.

  The second part, that Floer differential cylinders for the restricted data remain valid when considered in the larger space follows from standard arguments for split Hamiltonian systems and split almost complex structures (noting that $H_{t},J$ are both split in a neighborhood of $W\times \set{0}$). Briefly, one shows that the linearized operator at a solution $u=(u_{1},0)$ also splits into a diagonal matrix:
  \begin{equation*}
    D_{u}=\left[
      \begin{matrix}
        D_{u_{1}}&0\\
        0&\bar{\partial}+2\pi b/a
      \end{matrix}
    \right]
  \end{equation*}
  since $\bar{\partial}+2\pi b/a$ is an isomorphism, since $b\in (0,a)$, acting on the appropriate Sobolev spaces: $$W^{1,p}(\R\times \R/\Z,\C)\to L^{p}(\R\times \R/\Z,\C),$$ the regularity and rigidity of the solution $u_{1}$ is equivalent to the regularity and rigidity of the solution $u$.
\end{proof}

Due to Lemma \ref{lemma:floer-cylinders-KMAP}, we can pick $\delta>0$ and $\mathscr{U}$ small enough so that, for each $H_{t}\in \mathscr{O}_{b}(\delta,\mathscr{U})$, there is a tautological isomorphism of complexes:
\begin{equation*}
  \mathfrak{K}_{H_{t}}:\mathit{CF}(H_{t},J)\simeq \mathit{CF}(H_{t}|_{W\times \set{0}},J|_{W\times \set{0}}).
\end{equation*}
This map should be thought of as one representation of the Künneth map satisfying Theorem \ref{theorem:filtered-kunneth}; in the next section we explain how to make $\mathfrak{K}$ ``canonical,'' i.e., independent of the auxiliary choices used in $H_{t}$.

\subsubsection{Canonicity of the Künneth map}
\label{sec:canon-kunn-map}

Recall from \S\ref{sec:introduction} that $\mathit{SH}_{b}(Q)$ is defined to be the (formal) inverse limit of the Floer cohomologies of systems which agree with $b r_{Q}$ outside of a compact set. Thus, for $\delta>0$ and $\mathscr{U}$ small enough such that $H_{t}\in \mathscr{O}_{b}(\delta,\mathscr{U})$ satisfies Lemma \ref{lemma:floer-cylinders-KMAP} we have the following diagram:
\begin{equation}\label{eq:definition-of-K}
  \begin{tikzcd}
    \mathit{SH}_{b}(K_{\delta})\arrow[d]\arrow[r,dashed,"\mathfrak{K}_{\delta}"]&\mathit{SH}_{b}(Q)\arrow[d]\\
    \mathit{HF}(H_{t},J)\arrow[r,"\mathfrak{K}_{H_{t}}"]&\mathit{HF}(H_{t}|_{W\times \set{0}},J|_{W\times \set{0}})
  \end{tikzcd}
\end{equation}
where $K_{\delta}\subset Q\times D(a)$ is as in Definition \ref{definition:Kdelta}. The upper horizontal map is well-defined, for fixed $H_{t}$, by inverting the right vertical map. However, it a priori depends on the choice of $P_{t}$ used in the definition of $H_{t}$. In this section we show the result is independent of $P_{t}$.

\begin{lemma}\label{lemma:canonical-K}
  The upper horizontal map in \eqref{eq:definition-of-K} is independent of the choice of $H_{t}\in \mathscr{O}_{b}(\delta,\mathscr{U})$, provided that $\delta$ and $\mathscr{U}$ are sufficiently small.
\end{lemma}
\begin{proof}
  It suffices to prove that the maps $\mathfrak{K}_{H_{t}}$ commute with the continuation maps arising from changing $P_{t}$. This argument follows the same lines as Lemma \ref{lemma:floer-cylinders-KMAP}, except Floer cylinders are replaced by continuation cylinders. The details are left to the reader.
\end{proof}

\subsubsection{Compatibility of the Künneth map with the unit element}
\label{sec:compatible-KM-unit}

Fix $\delta$ and $\mathscr{U}$ small enough that Lemma \ref{lemma:canonical-K} applies; this fixes a domain $K_{\delta}$ and a map $\mathfrak{K}_{\delta}:\mathit{SH}_{b}(K_{\delta})\to \mathit{SH}_{b}(Q)$. The goal of this section is proving:
\begin{lemma}\label{lemma:PSS-KMAP}
  For $\delta$ small enough, the isomorphism $\mathfrak{K}_{\delta}:\mathit{SH}_{b}(K_{\delta})\to \mathit{SH}_{b}(Q)$ preserves the unit elements, i.e., $\mathfrak{K}_{\delta}(\mathit{PSS}(1))=\mathit{PSS}(1).$
\end{lemma}
\begin{proof}
  The proof is similar to the proof of Lemma \ref{lemma:floer-cylinders-KMAP} and \ref{lemma:canonical-K}; one shows (by a compactness argument) that all of the solutions contributing to $\mathit{PSS}(1)$ remain entirely in the submanifold $W\times \set{0}$.

  The argument that solutions contributing to $\mathit{PSS}(1)$ for the restricted data are rigid/regular if and only if they are rigid/regular for the unrestricted data follows a similar argument as the one used in the proof of Lemma \ref{lemma:floer-cylinders-KMAP}. However, here it is truly necessary to use that $b/a\in (0,1)$, while the proof of Lemma \ref{lemma:floer-cylinders-KMAP} technically only requires that $b/a\not\in \Z$. The reason is that the linearized operator for a PSS solution $u:\C\to W$ appears in the form:
  \begin{equation*}
    D_{u}=\left[
      \begin{matrix}
        D_{u_{1}}&0\\
        0&\bar{\partial}+\beta(-s)2\pi b/a
      \end{matrix}
    \right]
  \end{equation*}
  in cylindrical coordinates $z=e^{-2\pi(s+it)}$; here we agree to use the cut-off function $\beta$ from Definition \ref{definition:cut-off} in the definition of the PSS equation, which involves an interpolation between the zero Hamiltonian and $H_{t}$; see \S\ref{sec:pss-elements} and the references therein. If $b/a\in (0,1)$, then the operator:
  \begin{equation}\label{eq:PSS-operator}
    \bar{\partial}+\beta(-s)2\pi b/a:W^{1,p}(\C,\C)\to L^{p}(\C,\C)
  \end{equation}
  is an isomorphism, as can be established using a Fourier series argument. The argument then proceeds as in Lemma \ref{lemma:floer-cylinders-KMAP}.
\end{proof}

\begin{proof}[Proof of Theorem \ref{theorem:filtered-kunneth}]
  Fix $\delta$ and $\mathscr{U}$ small enough that Lemmas \ref{lemma:floer-cylinders-KMAP}, \ref{lemma:canonical-K}, and \ref{lemma:PSS-KMAP} apply; this fixes a domain $K=K_{\delta}$ and a map $\mathfrak{K}:\mathit{SH}_{b}(K)\to \mathit{SH}_{b}(Q)$. By what we have shown above, this map $\mathfrak{K}$ satisfies Theorem \ref{theorem:filtered-kunneth}.
\end{proof}

\subsection{Submersions and Legendrian lifts}
\label{sec:cotangent-submersion}

In this section we prove Theorem \ref{theorem:cotangent-submersion}. The proof is an application of simple facts about symplectic reductions. The first step is to recall that a smooth submersion $\pi:E\to M$, where $E$ is a compact manifold, is automatically a smooth proper fiber bundle (this result is often attributed to Ehresmann).

Let $\pi:E\to M$ be a smooth proper fiber bundle. Define:
\begin{equation*}
  N=\set{(p,q)\in T^{*}E:p\text{ vanishes on }\ker \d\pi};
\end{equation*}
then $N$ is a coisotropic submanifold of $T^{*}E$ (see \cite[\S5.4]{mcduff-salamon-book-2017}).
\begin{lemma}\label{lemma:char-fol}
  The leaves of the characteristic foliation of $N$ (that is to say, the integrable submanifolds tangent to the $\d\lambda$-complement to $TN$) consist of the submanifolds:
  \begin{equation}\label{eq:leaf-of-N}
    \mathscr{F}_{x,y}=\set{(p,q)\in T^{*}E:p=x\circ \d \pi \text{ and }\pi(q)=y},
  \end{equation}
  where $(x,y)\in T^{*}M$ is in the complement of the zero section.
\end{lemma}
\begin{proof}
  It is easy to see that the tangent space of each leaf $\mathscr{F}_{x,y}$ is contained in the $\d\lambda$-complement to $TN$. Since:
  \begin{itemize}
  \item $\dim \mathscr{F}_{x,y}=\dim E-\dim M$, and 
  \item $\dim N=\dim E+\dim M$,
  \end{itemize}
  we conclude from dimension reasons that $\dim \mathscr{F}_{x,y}$ equals the dimension of the characteristic foliation.
\end{proof}

Let $SN$ be the spherization\footnote{The reader is warned that the symbol ``$S$'' is used both for \emph{spherization} (dimension lowering operation applied to vector bundles) and \emph{symplectization} (dimension raising operation applied to cooriented contact manifolds). The meaning of the symbol should be inferred from the context.} of the vector bundle $N\to E$ and consider the following diagram:
\begin{equation}\label{eq:reduction-diagram}
  \begin{tikzcd}
    SN\arrow[r,"\subset"]\arrow[d,"\pi"]&ST^{*}E\\
    ST^{*}M.&
  \end{tikzcd}
\end{equation}
It is a standard fact of symplectic reductions that $\pi^{-1}(\Lambda)\subset SN$ is a Legendrian in $ST^{*}E$ for every Legendrian $\Lambda\subset ST^{*}M$.

\begin{lemma}
  The map $\pi$ in \eqref{eq:reduction-diagram} is proper. In particular, if $\Lambda$ is a compact Legendrian in $ST^{*}M$, then $\pi^{-1}(\Lambda)\subset N$ is a compact Legendrian in $ST^{*}E$.
\end{lemma}
\begin{proof}
  The leaves of the characteristic foliation on $N$ are the sets $\mathscr{F}_{x,y}$ appearing in \eqref{eq:leaf-of-N}; in particular, if $K\subset T^{*}M$ is a compact set, then $(p_{n},q_{n})\in \pi^{-1}(K)$ is a compact set, because if $(x_{n},y_{n})\in K$ converge then $q_{n}$ converge, after passing to a subsequence, (since $\pi:E\to M$ is proper and $y_{n}$ converges), and it follows automatically that $p_{n}=x_{n}\circ \d \pi$ converges.

  If $\Lambda\subset ST^{*}M$ is a compact set, then there is a compact set $K\subset T^{*}M-M$ whose projection to the ideal boundary $ST^{*}M$ is $\Lambda$. Then $\pi^{-1}(K)\subset N$ is a compact set, by the first paragraph of the proof. It is clear that $\pi^{-1}(K)$ is in $T^{*}E-E$, and it is easy to see that the ideal projection of $\pi^{-1}(K)$ equals $\pi^{-1}(\Lambda)\subset ST^{*}E$; thus $\pi^{-1}(\Lambda)$ is compact.
\end{proof}

The final step in the proof of Theorem \ref{theorem:cotangent-submersion} is to prove:
\begin{lemma}\label{lemma:uniformly-fast}
  If $\Lambda_{s}$ is a uniformly fast Legendrian isotopy in $ST^{*}M$ (see Definition \ref{definition:uniformly-fast}), then $\pi^{-1}(\Lambda_{s})$ is a uniformly fast Legendrian isotopy in $ST^{*}E$.
\end{lemma}
\begin{proof}
  Consider the reduction map $\pi:N\to T^{*}M$. We compute:
  \begin{equation}\label{eq:lucky-relation}
    \pi^{*}\lambda=p\d q;
  \end{equation}
  writing $\lambda=xdy$, this follows from the relationship between $(p,q)\in N$ and its projection $\pi(p,q)=(x,y)$ where $p=x\circ \d \pi$ and $y=\pi(q)$. The result then follows easily: let $K_{s}$ be some lift of $\Lambda_{s}$ to the symplectization $T^{*}M-M$. Because $\Lambda_{s}$ is uniformly fast, and because of \eqref{eq:lucky-relation}, the velocity vectors of curves contained in $\pi^{-1}(K_{s})$ are uniformly positive when inserted into $r^{-1}p\d q=r^{-1}\pi^{*}\lambda$, for any radial function $r:T^{*}E-E\to (0,\infty)$. This implies $\pi^{-1}(\Lambda_{s})$ is uniformly fast, as desired.
\end{proof}

\begin{proof}[Proof of Theorem {\ref{theorem:cotangent-submersion}}]
  Under the hypotheses of Theorem \ref{theorem:cotangent-submersion}, we conclude from Lemma \ref{lemma:uniformly-fast} that $\pi^{-1}(\Lambda_{s})$ must reintersect $\pi^{-1}(\Lambda_{0})$, by Theorem \ref{theorem:strong-main}, and thus $\Lambda_{s}$ must reintersect $\Lambda_{0}$, as desired.
\end{proof}

\begin{remark}\label{remark:ACC-fibration}
  The same argument proves the following: suppose $E$ is a compact manifold and $ST^{*}E$ solves the Arnol'd chord conjecture, in the sense that all compact Legendrians admit a non-constant Reeb chord for every choice of Reeb flow, and suppose there is a submersion $E\to M$. Then $ST^{*}M$ also solves the Arnol'd chord conjecture. This is because any Reeb flow on $ST^{*}M$ is generated by a positive $1$-homogeneous radial function $r_{M}$ on $T^{*}M$, which lifts to a positive $1$-homogeneous radial function $r_{M}\circ \pi$ on $N$, which can be extended (non-canonically) to a positive $1$-homogeneous radial function $r_{E}$ on $T^{*}E$. The ideal restriction of the Hamiltonian flow of $X_{r_{E}}$ is a Reeb flow which preserves $SN$ and is related to the Reeb flow corresponding to $r_{M}$ via the map $SN\to ST^{*}M$.
\end{remark}
\begin{remark}
  Remark \ref{remark:ACC-fibration} can be placed in a more general context. Suppose $Y_{1}$ and $Y_{2}$ are two compact cooriented contact manifolds, and suppose there is a coisotropic submanifold $N\subset SY_{1}$ which is invariant under the Liouville flow and which admits a \emph{proper} Liouville equivariant reduction map: $$\pi:N\to SY_{2},$$ where the word ``reduction'' means that the fibers of this map are the leaves of the characteristic foliation of $N$, and that $\pi^{*}\d \lambda_{2}=\d \lambda_{1}|_{N}$, where $\lambda_{i}$ is the canonical Liouville form on $SY_{i}$. It follows from Liouville equivariance that $\pi^{*}\lambda_{2}-\lambda_{1}=0$. The above argument then goes through and one concludes that $Y_{1}$ solving the Arnol'd chord conjecture implies $Y_{2}$ solves the Arnol'd chord conjecture.
\end{remark}

\subsection{String topology and products of spheres}
\label{sec:string-topology-products-spheres}

In this section we prove Theorem \ref{theorem:spheres} concerning classes in $H_{*}(\Lambda M)$ for $M=S^{n_{1}}\times \dots \times S^{n_{k}}$. First, we construct classes $A$ and $B$ for the case of $M = S^n$. The key part of this step is the notion of \emph{completing manifolds} as in \cite[\S 6]{oancea-EMS-2015} which relates classes in the spherical tangent bundle $STS^{n}$ with classes in $\Lambda S^{n}$. This step is done in \S\ref{sec:completing-manifolds}. In \S\ref{sec:generalized-section} we construct a special class in $H_{n}(STS^{n})$ using the notion of a generalized section. Afterwards in \S\ref{sec:product-spheres} we use the case of one sphere to construct the classes in $M=S^{n_{1}}\times \dots \times S^{n_{k}}$.

Throughout this section we use the bordism class approach to the homology of the loop space, as in \cite{brocic-cant-arXiv-2025}.

\subsubsection{Completing manifolds}
\label{sec:completing-manifolds}

Let $STS^n$ denote the spherical tangent bundle. Each point $v$ determines its \emph{orthogonal equator}, which is the unique equator passing through the basepoint of $v$ and which is orthogonal to $v$ at that basepoint. Consider the sphere bundle of pairs:
\begin{equation*}
  E=\set{(x,v)\in S^{n}\times STS^{n}:x\text{ lies in the orthogonal equator of }v}.
\end{equation*}
Define the map $F: E \to \Lambda S^n$ which sends $(x, v) \in E$ to the circle obtained by intersecting $S^n$ with the affine plane passing through $x$ and $v$ (this means the affine plane intersects the basepoint of $v$ and is tangent to $v$). This circle is parametrized by $\R/\Z$ with constant speed in the unique way that the tangent vector at $0$ points in the direction of $v$.

This procedure yields the \emph{completing manifolds operation}: $$\gamma:H_*(STS^n) \to H_{*+n-1} (\Lambda S^n)$$ as follows: given a class $f: X \to STS^n$, the class $\gamma f$ is the composition of the two upper horizontal morphisms in the following diagram:
\begin{equation*}
  \begin{tikzcd}
    f^* E \arrow[r]\arrow[d] &  E\arrow[d, "\pi"] \arrow[r]  & \Lambda S^n \\
    X \arrow[r, "f"] & STS^n.
  \end{tikzcd}
\end{equation*}

\subsubsection{Generalized section}
\label{sec:generalized-section}

In this section we construct a special class:
\begin{equation*}
  s: \R P^{n} \to STS^{n}
\end{equation*}
to which we will apply the completing manifolds construction of \S\ref{sec:completing-manifolds} to obtain a class in $H_{2n-1}(\Lambda S^{n})$. The construction of $s$ is as a \emph{generalized section}, i.e., that $\pi \circ s : \R P^n \to S^n$ has degree $1$. Even though $STS^{n}\to S^{n}$ does not admit a section for even values of $n$, it does admit a generalized section: after a blow up of the base space, transforming $S^{n}$ to $\R P^{n}$, it admits a section. We describe this with greater detail.

The first step is to recall the \emph{unoriented blow-down map} $b:\R P^{n}\to S^{n}$.
\begin{lemma}\label{lemma:unorient-blow-down}
  Let $N\in S^{n}$ denote the north pole, and let $\mathit{stereo}:S^{n}\setminus N\to \R^{n}$ be the stereographic projection. There is a unique smooth map $b:\R P^{n}\to S^{n}$ such that:
  \begin{equation*}
    \begin{tikzcd}
      \R^{n}\arrow[r]\arrow[d,"\mathit{stereo}^{-1}"]&\R P^{n}\arrow[d,"b"]\\
      S^{n}\setminus N\arrow[r] &S^{n}
    \end{tikzcd}
  \end{equation*}
  where the upper horizontal map is the parametrization:
  \begin{equation*}
    (x_{1},\dots,x_{n})\mapsto [1:x_{1}:\dots:x_{n}]
  \end{equation*}
  associated to the zeroth affine coordinate chart on $\R P^{n}$.
\end{lemma}
\begin{proof}
  Uniqueness is obvious since the domain of the zeroth affine chart is dense. Moreover, the map obviously must extend by collapsing the hyperplane divisor at infinity corresponding to $x_{0}=0$ to the north pole $N$.

  It remains only to check this map is actually smooth. First recall that the inverse of the stereographic projection map is given by:
  \begin{equation}\label{eq:stereo-inverse}
    \mathit{stereo}^{-1}(x_{1},\dots,x_{n}):=\frac{(x_{1}^{2}+\dots+x_{n}^{2}-1,2x_{1},\dots,2x_{n})}{x_{1}^{2}+\dots+x_{n}^{2}+1}
  \end{equation}
  To prove this extends smoothly from the zeroth affine coordinate chart to all of $\R P^{n}$, we check how the formula transforms when we move to the $k$th affine chart parametrized by $[y_{0}:\dots:y_{k-1}:1:y_{k+1}:\dots:y_{n}]$; the transition function is $x_{j}=y_{j}/y_{0}$, with the convention that $y_{k}=1$. In these new coordinates, \eqref{eq:stereo-inverse} is given by:
  \begin{equation}\label{eq:blow-down-in-affine-chart}
    \frac{(y_{1}^{2}+\dots+y_{n}^{2}-y_{0}^{2},2y_{1}y_{0},\dots,2y_{n}y_{0})}{y_{0}^{2}+y_{1}^{2}+\dots+y_{n}^{2}},
  \end{equation}
  which evidently extends smoothly to the entire affine chart (it was initially only defined when $y_{0}\ne 0$, as that is where the $k$th affine chart intersects the zeroth affine chart). Since $k$ was arbitrary, we conclude the result.
\end{proof}

\begin{lemma}
  There exists a smooth map $s:\R P^{n}\to STS^{n}$ such that:
  \begin{equation*}
    \begin{tikzcd}
      &STS^{n}\arrow[d,"\pi"]\\
      \R P^{n}\arrow[ru,out=90,in=180,"s"]\arrow[r,"b"]&S^{n}.
    \end{tikzcd}
  \end{equation*}
  In the zeroth affine chart, $s$ can be chosen to agree with the push-forward of a constant vector field on $\R^{n}$ via the inverse of the stereographic projection.
\end{lemma}
\begin{proof}
  The derivative of $\mathit{stereo}^{-1}(x)$ with respect to the $x_{1}$ variable extends to the following vector field on $S^{n}$:
  \begin{equation*}
    V(q)=(q_{1}(1-q_{0}),1-q_{0}-q_{1}^{2},-q_{1}q_{2},\dots,-q_{1}q_{n}).
  \end{equation*}
  We will prove that the ray spanned by $V\circ b$ extends smoothly to a map $s$.

  In the affine chart \eqref{eq:blow-down-in-affine-chart} with $k>0$, one computes:
  \begin{equation*}
    V\circ b=\frac{(2y_{1}y_{0},y_{0}^{2}-y_{1}^{2}+y_{2}^{2}+\dots,2y_{1}y_{2},\dots,2y_{1}y_{n})}{(2y_{0}^{2})^{-1}(y_{0}^{2}+y_{1}^{2}+\dots+y_{n}^{2})^{2}}
  \end{equation*}
  Thus we can set:
  \begin{equation*}
    s([y_{0}:\dots:y_{n}])=\R_{+}(2y_{1}y_{0},y_{0}^{2}-y_{1}^{2}+y_{2}^{2}+\dots,2y_{1}y_{2},\dots,2y_{1}y_{n}).
  \end{equation*}
  This ray is well-defined; indeed, if the generator vanishes, then:
  \begin{equation}\label{eq:to-be-contradicted}
    y_{1}^{2}=y_{0}^{2}+y_{2}^{2}+\dots+y_{n}^{2}\implies y_{1}\ne 0,
  \end{equation}
  and the other entries $2y_{1}y_{i}=0$ yield $y_{0}=y_{2}=\dots=y_{n}=0$ which contradicts \eqref{eq:to-be-contradicted}. This completes the proof.
\end{proof}

\subsubsection{Manipulation of classes in the free loop space of a sphere}
\label{sec:manip-class-free}

Apply the completing manifolds operation from \S\ref{sec:completing-manifolds} to the map $s: \R P^n \to STS^n$ constructed in \S\ref{sec:generalized-section} to produce the class $B=\gamma s\in H_{2n-1}(\Lambda S^{n})$. Similarly, let $P\in H_{n-1}(\Lambda S^{n})$ be the class obtained from the point class in $STS^{n}$, and let $\mathit{pt}$ denote the point class in $H_{0}(\Lambda S^{n})$ consisting of a single constant loop. Our first lemma in this section is:
\begin{lemma}\label{lemma:B-ast-pt-P}
  It holds that $B\ast \mathit{pt}=P$, where $\ast$ denotes the Chas-Sullivan product.
\end{lemma}
\begin{proof}
  If we pick the representative of $\mathit{pt}$ which is a constant loop not based at the north pole $N$, this follows easily from the definition of the Chas-Sullivan product in terms of transverse intersections; see \cite[\S 2.1]{brocic-cant-arXiv-2025}.

  Briefly, $B\ast \mathit{pt}$ considers those loops in $B$ whose base points are $\mathit{pt}$; since:
  \begin{itemize}
  \item the base points of $B$ cover the blow-down map $b$, 
  \item $\mathit{pt}$ is a regular value of $b$, and,
  \item there is a unique point in $b^{-1}(\mathit{pt})$,
  \end{itemize}
  we conclude that $B\ast \mathit{pt}$ has the same domain as $P$, namely $E|_{\mathit{pt}}$; see \S\ref{sec:completing-manifolds}. The only difference is that loops in $P$ are parametrized with unit speed while loops in $B\ast \mathit{pt}$ are parametrized differently; nonetheless, the two classes are equal in $H_{n-1}(\Lambda S^{n})$.
\end{proof}
\begin{remark}
  In the case $n=1$, there are two connected components of $STS^{n}$, and so $P$ is not a well-defined class, but rather there are two classes depending on the orientation (and $B\ast \mathit{pt}$ will equal one of these, depending on how the section $s$ is chosen).
\end{remark}

Following the terminology in \cite[\S2.2]{brocic-cant-arXiv-2025} introduce the \emph{action class}:
\begin{equation*}
  A(x;t)=(\cos(2\pi t)x_{0}-\sin(2\pi t)x_{1},\sin(2\pi t)x_{0}+\cos(2\pi t)x_{1},x_{2},\dots,x_{n}),
\end{equation*}
associated to the rotation circle action so that $A\in H_{n}(\Lambda S^{n})$, and similarly let $A^{-1}(x;t)=A(x;-t)$. We have:
\begin{lemma}\label{lemma:A-A-inv}
  We have $A\ast A^{-1}=[S^{n}]$. In the case $n>1$, we have $A^{-1}=A$.
\end{lemma}
\begin{proof}
  The first part is a direct consequence of \cite[\S2.2.1]{brocic-cant-arXiv-2025}. The second part follows from the fact that $\pi_{1}(\mathit{SO}(n+1))=\Z/2\Z$ for $n>1$.
\end{proof}
\begin{lemma}\label{lemma:delta-P-A-inv}
  It holds that $\Delta P=A^{-1}$ in $H_{n}(\Lambda S^{n})$; in the case $n=1$ we require that $P$ is obtained by picking a point in the positively oriented component of $STS^{1}$.
\end{lemma}
\begin{proof}
  In \cite[\S2.2.2]{brocic-cant-arXiv-2025} we construct an explicit class $D\in H_{n-1}(\Lambda S^{n})$ using an open book decomposition of $S^{n}$ and show that $\Delta D=A^{-1}$. In this proof we will show that the class $P$ is homotopic to the class $D$, and so the desired result will follow.

  Fix $v\in STS^{n}$ and let $E_{v}$ be the equator orthogonal to $v$, as in \S\ref{sec:completing-manifolds}. Let us think of $v\in STS^{n}_{q}$ as a pair $(q,\R_{+}w)$ where $\R_{+}w$ is the ray through the origin spanned by a unit vector $w$ and $q$ is a basepoint on $S^{n}$.

  For any two distinct points in $x,y\in E_{v}$ there is a unique two-dimensional subspace $\Pi$ which contains $w$ and the vector $x-y$. Let us denote by:
  \begin{equation*}
    \Gamma(x,y)=(x+\Pi)\cap S^{n}=(y+\Pi)\cap S^{n}
  \end{equation*}
  parametrized as a circle $\R/\Z\to S^{n}$ via the formula:
  \begin{equation*}
    t\mapsto \frac{(x+y)+\cos(t)(x-y)}{2}+\frac{\sin(t)\abs{x-y}}{2}w.
  \end{equation*}
  with constant speed $2^{-1}\abs{x-y}$. By the completing manifolds construction, the class $P$ is represented by the map:
  \begin{equation}\label{eq:explicit-P}
    P:x\in E_{v}\mapsto \Gamma(q,x).
  \end{equation}
  For concreteness, let us fix $w=\partial_{0}$, $q=(0,1,0,\dots,0)$, so $E_{v}=\set{0}\times S^{n-1}$. Define:
  \begin{equation*}
    \ell(0,x_{1},\dots,x_{n})=x_{1}\qquad \rho(x_{0},\dots,x_{n})=(0,-x_{1},x_{2},\dots,x_{n})
  \end{equation*}
  
  The class $D$ considered in \cite[\S2.2.2]{brocic-cant-arXiv-2025} is (a smoothing of) the map:
  \begin{equation}\label{eq:explicit-D}
    D:x\in E_{v}\mapsto \left\{
      \begin{aligned}
        &\Gamma(\rho(x),x)&&\text{ if }\ell(x)\le 0,\\
        &x&&\text{ if }\ell(x)\ge 0.
      \end{aligned}
      \right.
    \end{equation}
    This map is not smooth at the interface $\ell(x)=0$ but it is continuous, and therefore represents a well-defined bordism class. The class parametrizes all loops passing through one page of an open book decomposition; we require the piecewise smooth formula to switch between the circle action on one page and the constant loops on the opposite page.

    \begin{claim}
      The classes $P$ and $D$ given in \eqref{eq:explicit-P} and \eqref{eq:explicit-D} are homotopic as maps $E_{v}\to \Lambda S^{n}$.
    \end{claim}
    \begin{proof}[Proof of claim]
      We construct a continuous homotopy $G:E_{v}\times [0,1]\to \Lambda S^{n}$ interpolating between $D$ and $P$ at the two extremities (one obtains a smooth homotopy between $P$ and a smoothing of $D$ by standard smooth approximation results). We define:
      \begin{equation}\label{eq:explicit-G}
        G:(x,r)\in E_{v}\in [0,1]\mapsto \left\{
          \begin{aligned}
            &\Gamma(\rho_{r}(x),x)&&\text{ if }\ell(x)\le r,\\
            &x&&\text{ if }\ell(x)\ge r,
          \end{aligned}
        \right.        
      \end{equation}
      where $\rho_{r}:E_{v}\cap \set{\ell \le r}\to E_{v}\cap \set{\ell \ge r}$ is any continuous map which satisfies $\rho_{r}(x)=x$ if $\ell(x)=r$ and $\rho_{0}(x)=\rho(x)$ and $\rho_{1}(x)=q$ for all $x$. We fix: $$\rho_{r}(0,x_{1},\dots,x_{n}) = (0,\beta_{r}(x_{1}),\alpha_{r}(x_{1})x_{2},\dots,\alpha_{r}(x_{1})x_{n})$$ where:      
      \begin{equation*}
        \beta_{r}(x) = \frac{x+1}{r+1}(r-1) +1
      \end{equation*}
      and $\alpha_{r}$ is determined by the condition that $\abs{\rho_r(0,x_{1},\dots,x_{n})}=1$. Then $G$ is the desired homotopy with $G(x,0)=D(x)$ and $G(x,1)=P(x)$; see Figure \ref{fig:bordism}.
    \end{proof}
    Since we have shown that $\Delta D=A^{-1}$ in \cite[Lemma 17]{brocic-cant-arXiv-2025}, it follows from the claim that $\Delta P=A^{-1}$, as desired.
\end{proof}
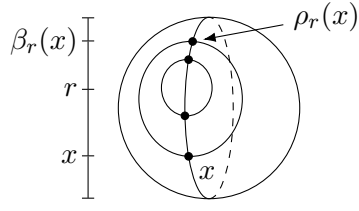
\begin{figure}[h!]
  \begin{tikzpicture}[scale=.8]
    \begin{scope}[shift={(4,0)}]
      \draw[|-] (-3,-1.5)--(-3,-0.8); \draw[|-] (-3,-0.8)--(-3,0.3); \draw[|-] (-3, 0.3) -- (-3, 1.1); \draw[|-|] (-3,1.1)--(-3,1.5);

      \path node at (-3, 0.3) [left] {$r$} node at (-3, -0.8) [left] {$x$} node at (-3,1.1) [left] {$\beta_{r}(x)$};
    \end{scope}
    \draw (3,0) circle (1.5); \draw[name path= el_left] (3,1.5) arc (90:270:0.4 and 1.5);

    \draw[dashed, name path=el_right] (3,-1.5) arc (-90:90:0.4 and 1.5);

    \coordinate (A) at (2, -0.8);

    \coordinate (B) at (4, -0.8);

    \coordinate (C) at (2,1.1);

    \coordinate (D) at (4,1.1);

    \path[name path = A--B] (A)--(B);

    \path[name path = C--D] (C)--(D);

    \path[name path = line_fx] ($(C)-(0,0.3)$) -- ($(D)-(0,0.3)$);

    \path[name path = line_x] ($(A)+(0,0.1)$) -- ($(D)+(0,0.1)$);

    \path[name intersections = {of=el_left and A--B, by=xn}];

    \path[name intersections = {of=el_left and C--D, by=fxn}];

    \path[name intersections = {of=el_left and line_x, by=x'}];

    \path[name intersections = {of=el_left and line_fx, by=fx'}];

    \node[draw, circle, fill=black, inner sep=1pt] at (xn) {};

    \node[draw, circle, fill=black, inner sep=1pt] at (fxn) {};

    \node (B) at (fxn) {};

    \node[draw, circle, fill=black, inner sep=1pt] at (x') {};

    \node[draw, circle, fill=black, inner sep=1pt] at (fx') {};

    \node at (xn) [below right] {$x$};

    \node (A) at ($(fxn)+(1.5,0)$) [above right] {$\rho_r(x)$};

    \draw[-Latex] (A)--(B);

    \coordinate (E) at ($0.5*(xn) - 0.5*(fxn)$);

    \draw let \p1=($0.5*(xn) - 0.5*(fxn)$), \n1={veclen(\x1,\y1)} in ($0.5*(xn)+0.5*(fxn)$) ellipse (0.9*\n1 and \n1);

    \draw let \p2=($0.5*(x') - 0.5*(fx')$), \n2={veclen(\x2,\y2)} in ($0.5*(x')+0.5*(fx')$) ellipse (0.9*\n2 and \n2);
  \end{tikzpicture}
  \caption{An illustration of the loop $G(x,r)$ with $x_{1}<r$.}
  \label{fig:bordism}
\end{figure}

This concludes the first part of Theorem \ref{theorem:spheres}, since $A\ast \Delta(B\ast \mathit{pt})=[S^{n}]$ by combining Lemmas \ref{lemma:B-ast-pt-P}, \ref{lemma:A-A-inv}, and \ref{lemma:delta-P-A-inv}.

\subsubsection{The product of spheres}
\label{sec:product-spheres}

For the second part on the product of spheres, we require the following notion:
\begin{definition}\label{definition:cartesian-product}
  Let $F_{1}:X_{1}^{n_{1}}\times \R/\Z\to M_{1}$ and $F_{2}:X_{2}^{n_{2}}\times \R/\Z\to M_{2}$ represent classes in $H_{n_{1}}(\Lambda M_{1})$ and $H_{n_{2}}(\Lambda M_{2})$ respectively. Then we define:
  \begin{equation*}
    F_{1}\times F_{2}:(X_{1}\times X_{2})\times \R/\Z\to M_{1}\times M_{2}
  \end{equation*}
  by the formula:
  \begin{equation*}
    (F_{1}\times F_{2})(x_{1},x_{2};t)=(F_{1}(x_{1};t),F_{2}(x_{2};t)),
  \end{equation*}
  so that $F_{1}\times F_{2}\in H_{n_{1}+n_{2}}(\Lambda(M_{1}\times M_{2}))$.
\end{definition}

\begin{lemma}\label{lemma:cartesian-product}
  Given classes $F_{1},F_{2}$ as in Definition \ref{definition:cartesian-product}, then:
  \begin{itemize}
  \item $\Delta(F_{1} \times F_{2})=F_{1}\times \Delta F_{2}$ provided that $F_{1}(x,t)=F_{1}(x,0)$, i.e., $F_{1}$ is a class of constant loops.
  \end{itemize}
  Given also $F_{1}',F_{2}'$ as in Definition \ref{definition:cartesian-product}, then:
  \begin{itemize}
  \item $(F_{1}\times F_{2})\ast (F_{1}'\times F_{2}')=(F_{1}\ast F_{1}')\times (F_{2}\ast F_{2}')$.
  \end{itemize}
\end{lemma}
\begin{proof}
  This follows easily from the definition of the BV operator and the Chas-Sullivan product; see \cite[\S 2.1]{brocic-cant-arXiv-2025}.
\end{proof}

\begin{proof}[{Completion of proof of Theorem \ref{theorem:spheres}}]
  Set:
  \begin{itemize}
  \item $A_i:=S^{n_{1}} \times \cdots \times A_{n_i} \times \cdots \times S^{n_{k}},$ and
  \item $B_i:=S^{n_{1}} \times \cdots \times B_{n_i} \times \cdots \times S^{n_{k}}.$
  \end{itemize}
  The classes $A_{n_i} \in H_{n_{i}}(\Lambda S^{n_{i}})$, and $B_{n_i} \in H_{2n_{i}-1}(S^{n_{i}})$ are from the construction in \S\ref{sec:manip-class-free} and $S^{n_{i}}$ is, by abuse of notation, the class of constant loops corresponding to the sphere $S^{n_{i}}$.

  It follows from Lemma \ref{lemma:cartesian-product} and the results of \S\ref{sec:manip-class-free} that:
  \begin{itemize}
  \item $C_i = [S^{n_{1}} \times S^{n_{2}} \times \cdots \times S^{n_i} \times \mathit{pt}\times \mathit{pt}\times \cdots]$,
  \end{itemize}
  where $C_{i}$ is defined in terms of the $A_{1},\dots,A_{i}$ and $B_{1},\dots,B_{i}$ as in the statement of Theorem \ref{theorem:spheres}.
\end{proof}

\subsection{Rationality constants of Lagrangians in cotangent bundles of tori}
\label{sec:proof-rationality}

In this section we prove Theorem \ref{theorem:rationality}. Fix a fiberwise starshaped subdomain $\Omega\subset T^{*}L$, consider a compact Lagrangian submanifold $L$ in the interior of $\Omega$. The goal is to prove that either $L$ is exact, or $\rho(L)$ is bounded above by a constant which depends only on $\Omega$.

\subsubsection{Non-exact Weinstein neighbourhood}
\label{sec:non-exact-weinstein}

The first step is to use the Weinstein neighborhood theorem to conclude a symplectic embedding $\iota:K\to \Omega$ of a disk cotangent bundle $K=DT^{*}L$ which extends the embedding $L\to \Omega$. It follows that the period group of the closed one-form:
\begin{equation*}
  \iota^{*}\lambda_{\Omega}-\lambda_{K}
\end{equation*}
is valued in $\rho(L) \Z$. This holds because $K$ deformation retracts onto $L$ and so the periods of $\iota^{*}\lambda_{\Omega}-\lambda_{K}$ are the same as the periods of $\lambda_{\Omega}|_{L}$.

\subsubsection{Classes in the free-loop space of cotangent bundle of a torus}
\label{sec:classes-free-loop}

The next step introduces certain classes in $\mathit{SH}(T^{*}T^{n})$; these classes were alluded to in Remark \ref{remark:on-the-constant-C}.

\begin{lemma}
  There exist classes $A_{i}^{\pm}\in \mathit{SH}(T^{*}T^{n})$, $i=1,\dots,n$ such that
  \begin{equation*}
    A_{i}^{+}\ast A_{i}^{-}=\mathit{PSS}(1)
  \end{equation*}
  and such that $A_{i}^{\pm}$ lies in the free homotopy class of $t\mapsto \pm te_{i}$.
\end{lemma}
\begin{proof}
  This is well-known, and follows from the arguments of \S\ref{sec:manip-class-free} regarding action classes, the arguments in \S\ref{sec:product-spheres} regarding cartesian products, and the Viterbo isomorphism (using the framework of \cite{brocic-cant-arXiv-2025}).
\end{proof}
Therefore, having fixed $\Omega\subset T^{*}T^{n}$, we can pick the slope $c>0$ large enough that $A_{i}^{+}\ast A_{i}^{-}=\mathit{PSS}(1)$ holds in $\mathit{SH}_{c}(\Omega)$ for each $i=1,\dots,n$.

\subsubsection{Truncated Viterbo restriction}
\label{sec:trunc-viterbo-restr}

Combining the set-up of \S\ref{sec:non-exact-weinstein} and \S\ref{sec:classes-free-loop}, we can apply Theorem \ref{theorem:VR-trunc} to conclude:
\begin{corollary}\label{corollary:the-B-classes}
  If the Lagrangian $L\subset \Omega$, as in \S\ref{sec:non-exact-weinstein}, has $\rho(L)>2c$, where $c$ is as in \S\ref{sec:classes-free-loop}, then there are classes $B_{i}^{\pm}=\mathfrak{R}(A_{i}^{\pm})$
  so that:
  \begin{itemize}
  \item $B_{i}^{+}\ast B_{i}^{-}=\mathit{PSS}(1)$
  \item $B_{i}^{\pm}\in \mathit{SH}(T^{*}L,\kappa_{i}^{\pm})$, where $\kappa_{i}^{\pm}$ is the collection of free homotopy classes of loops $\gamma$ lying in the kernel of $\iota^{*}\lambda_{\Omega}-\lambda_{K}$ and satisfying that $\iota(\gamma)$ lies in the free homotopy class of $t\mapsto \pm te_{i}$.
  \end{itemize}
\end{corollary}

\subsubsection{From symplectic cohomology back to string topology}
\label{sec:SH-to-string-topol}

In this subsection, we import one of the main results from \cite{abbondandolo-schwarz-CPAM-2006,abouzaid-EMS-2015} that there is a BV algebra map from $\mathit{SH}(T^{*}L)$ to $\mathit{HM}_{*}(\Lambda L)$. We emphasize here that $\mathit{HM}_{*}(\Lambda L)$ uses a Morse model rather than the bordism classes model used in \S\ref{sec:string-topology-products-spheres}.
\begin{theorem}\label{theorem:import}
  There is a graded BV-algebra $\mathit{HM}_{*}(\Lambda L)$ over $H^{*}(T^{*}L)$ with:
  \begin{itemize}
  \item the BV operator has degree $+1$,
  \item the product has degree $-n$,
  \item the algebra map sends $H^{*}(T^{*}L)$ to $\mathit{HM}_{n-*}(\Lambda L)$.
  \item there is a direct sum decomposition into summands $\mathit{HM}_{*}(\Lambda M,\kappa')$ corresponding to free homotopy classes $\kappa'$ of loops in $M$.
  \end{itemize}
   The group $\mathit{HM}_{d}(\Lambda L,\kappa')$ is isomorphic to the degree $d$ summand $H_{d}(\Lambda_{\kappa'} L)$ of the singular homology of the component $\Lambda_{\kappa'}L$ of free loop space $\Lambda L$.

  There is a map of BV algebras:
  \begin{equation*}
    \Psi:\mathit{SH}(T^{*}L)\to \mathit{HM}_{*}(\Lambda L),
  \end{equation*}
  defined by counting half-infinite Floer cylinders with boundary conditions on the zero section; the map $\Psi$ respects the BV operator, the product, and the algebra structure over $H^{*}(T^{*}L)$; the operation also respects free homotopy classes, in that $\mathit{SH}(T^{*}L,\kappa)$ is mapped into $\mathit{HM}_{*}(\Lambda L,\pi(\kappa))$ where $\pi$ is the obvious projection.
\end{theorem}
\begin{proof}
  The construction of this map is the main topic of \cite{abouzaid-EMS-2015}. We comment that the algebra map $H^{*}(T^{*}L)\to \mathit{HM}_{n-*}(\Lambda L)$ is given by a Thom isomorphism style map (take a class in $H^{*}(T^{*}L)$, intersect it with the zero section, and take the corresponding cycle of constant loops).  
\end{proof}

Referring to the notation and conclusions of Corollary \ref{corollary:the-B-classes}, we conclude:
\begin{corollary}\label{corollary:the-C-classes}
  Assume the context of Corollary \ref{corollary:the-B-classes}; there are non-zero classes $C_{i}\in H_{*}(\Lambda_{\kappa_{i}} L)$
  so that:
  \begin{itemize}
  \item $\deg(C_{i})\ge n$,
  \item $\kappa_{i}$ is the collection of free homotopy classes of loops $\gamma$ lying in the kernel of $\iota^{*}\lambda_{\Omega}-\lambda_{K}$ and satisfying that $\iota(\gamma)$ lies in the free homotopy class of $t\mapsto \pm te_{i}$, for some choice of $\pm$ sign.
  \end{itemize}
\end{corollary}
\begin{proof}
  Let $C_{i}^{\pm}=\Psi(B_{i}^{\pm})$ where $\Psi$ is as in Theorem \ref{theorem:import}. Since $C_{i}^{+}\ast C_{i}^{-}$ equals the unit element of degree $n$, we conclude that one of $C_{i}^{+}$ or $C_{i}^{-}$ must have degree at least $n$. Let us suppose the sign is $\epsilon\in \set{+,-}$. Let $\kappa_{i}$ be the image of the corresponding $\kappa_{i}^{\epsilon}$ under $\pi:T^{*}L\to L$, and set $C_{i}$ equal to the image of $C_{i}^{\epsilon}$ under the identification between $\mathit{HM}_{*}(\Lambda L,\kappa_{i})$ and  $H_{*}(\Lambda_{\kappa_{i}} L)$. It follows from the construction that $C_{i},\kappa_{i}$ satisfy the statement.
\end{proof}

\subsubsection{Two facts about aspherical manifolds}
\label{sec:two-facts-about}

We recall two facts about the string topology aspherical manifolds taken directly from \cite[\S 5]{latschev-EMS-2015}.
\begin{lemma}[{\cite[Lemma 5.1]{latschev-EMS-2015}}]\label{lemma:lat1}
  Let $L$ be a connected manifold, and let $Z$ be the centralizer of an element $[\gamma]\in \pi_{1}(L,\mathit{pt})$, and let $\pi:L'\to L$ be the covering space associated to the subgroup $Z$, and let $\gamma'$ be a lift of $\gamma$. Then the induced projection $\pi:\Lambda_{\gamma'}L'\to \Lambda_{\gamma}L$ is a homeomorphism.\hfill$\square$
\end{lemma}
\begin{lemma}[{\cite[Lemma 5.2]{latschev-EMS-2015}}]\label{lemma:lat2}
  Let $L$ be a connected aspherical manifold, and let $L',\gamma'$ be as in Lemma \ref{lemma:lat2}. Then the evaluation at the basepoint map $\Lambda_{\gamma'}L'\to L'$ is a homotopy equivalence.\hfill$\square$
\end{lemma}

We apply these lemmas in our sitation. Let $C_{i},\kappa_{i}$ be as in Corollary \ref{corollary:the-C-classes}. Pick elements $\gamma_{i}\in \kappa_{i}$ based at a chosen basepoint $\mathit{pt}\in L$ so that we may consider $[\gamma_{i}]\in \pi_{1}(L,\mathit{pt})$.

\begin{lemma}\label{lemma:zn_subrgroup}
  In the above context, and assuming that the Lagrangian $L$ is aspherical, then there are powers $p_{i}\ge 1$ such that $[\gamma_{i}]^{p_{i}}$ pairwise commute for $i=1,\dots,n$.
\end{lemma}
\begin{proof}
  The proof has two steps; first we show the centralizer $Z_{i}$ of $[\gamma_{i}]$ is a finite index subgroup. From Lemma \ref{lemma:lat1}, the projection $\Lambda_{\gamma'_i} L' \to \Lambda_{\gamma_i} L$ is a homeomorphism, where $\Pi:L' \to L$ is a covering associated to $Z$. Because the non-zero class $C_{i}$ has degree at least $n$, and $C_{i}$ lifts to a class in $\Lambda_{\gamma'_{i}}L'$, Lemma \ref{lemma:lat2} implies $L'$ has a non-vanishing homology in some degree at least $n$. Since $L'$ is a closed $n$ manifold, this implies that $L'$ is compact, and so the covering $L' \to L$ is finite. Consequently, $Z_{i}$ is a finite-index subgroup.

  The second step is to show there is $p_i$ such that $[\gamma_i]^{p_i} \in Z_{1}\cap \dots \cap Z_{n}$. Set $q_j:= [\pi_1(L) : Z_{j}]$; then for each element $g \in \pi_1(L,\mathit{pt})$ we have that $g^{q_j} \in Z_{j}$. Setting $p_i= \prod_{j\neq i} q_j$ completes the proof.
\end{proof}

\begin{lemma}\label{lemma:Zn-subgroup-main}
  Assume the context of Lemma \ref{lemma:zn_subrgroup} (which assumes there is an aspherical Lagrangian $L\subset \Omega$, as in \S\ref{sec:non-exact-weinstein}, with $\rho(L)>2c$, where $c$ is as in \S\ref{sec:classes-free-loop}). The subgroup:
  \begin{equation*}
    G=\set{[\gamma_1]^{p_{1}k_1} \cdots [\gamma_n]^{p_{n}k_n}: k_i \in \Z}
  \end{equation*}
  is a finite index subgroup of $\pi_1(L)$ which is isomorphic to $\Z^{n}$, and which lies in the kernel of the 1-form $\iota^{*}\lambda_{\Omega}$ where $\iota:L\to \Omega$ is the inclusion.
\end{lemma}
\begin{proof}
  To prove that $G$ is isomorphic to $\Z^{n}$, it suffices to show that:
  \begin{equation*}
    [\gamma_1]^{p_{1}k_1} \cdots [\gamma_n]^{p_{n}k_n}=1\implies k_{1}=\dots=k_{n}=0.
  \end{equation*}
  This holds since $\gamma_{i}\in \kappa_{i}$ and $\iota(\kappa_{i})$ is in the free homotopy class of $\pm e_{i}$, so:
  \begin{equation*}
    \iota_{*}([\gamma_1]^{p_{1}k_1} \cdots [\gamma_n]^{p_{n}k_n})
  \end{equation*}
  lies in a non-trivial free homotopy class.

  To prove that $G$ is finite index, we argue as follows: since $L$ is aspherical, so is any covering space $L' \to L$. Let $L'$ be the covering space associated with the subgroup $G$. Since $G \cong \Z^n$ we have that $L'$ is an Eilenberg-MacLane space $K(\Z^n, 1)$, which further implies that $L'$ is homotopy equivalent to $T^n$. From this, we conclude that $L'$ is a closed manifold and hence the covering $L' \to L$ is finite.

  The last part of the statement follows from the fact that each $\kappa_{i}$ lies in the kernel of $\iota^{*}\lambda_{\Omega}$, as established in Corollary \ref{corollary:the-C-classes}.
\end{proof}

\subsubsection{Proof of Theorem \ref{theorem:rationality}}
\label{sec:proof-theorem-rationality}

By Lemma \ref{lemma:Zn-subgroup-main}, if $\iota:L\to \Omega$ is a Lagrangian embedding with $\rho(L)>2c$ where $c$ is as in \S\ref{sec:classes-free-loop}, then the kernel of $\iota^{*}\lambda_{\Omega}$ contains a finite index subgroup $G$. The kernel of a non-zero homomorphism to $\R$ has an infinite index. This implies $\iota^{*}\lambda_{\Omega}=0$, as desired.\hfill$\square$

\appendix

\section{Viterbo restriction}
\label{sec:viterbo-restriction}

The \emph{Viterbo restriction map}, as discussed in \cite{viterbo-GAFA-1999,mclean-GT-2009,abouzaid-seidel-GT-2010,ritter-jtopol-2013,zhou-JSG-2022,guo-zhou-arXiv-2026}, is a map $\mathfrak{R}:\mathit{SH}(\Omega)\to \mathit{SH}
(K)$ associated to an exact embedding $\iota:K\to \Omega$ of Liouville domains. This map can be refined to a filtered version, and this refinement respects additional structures; the precise statement we require in the body of the text is given in Theorem \ref{theorem:VR-main}. The primary goal of this section is proving this result. Since it is quite well-understood, we have opted to place this part of the paper in an appendix.

As developed in \cite{zhou-JSG-2022} (and also used in \cite{guo-zhou-arXiv-2026}), there is also a ``truncated Viterbo restriction map'' for non-exact embeddings. We required a version of this in the proof of Theorem \ref{theorem:rationality}; the precise statement we used was given in Theorem \ref{theorem:VR-trunc} and we briefly review the proof in \S\ref{sec:truncated-viterbo-restriction}

We recall the geometric set-up for Theorem \ref{theorem:VR-main}: we let $\iota:K\to \Omega$ be an exact embedding of Liouville domains, and suppose that the image of $K$ is contained in the interior of $\Omega$. To be precise, this means that $K$ is a Liouville domain with its own Liouville form $\lambda_{K}$, and the difference $\lambda_{K}-\iota^{*}\lambda$ is an exact form. When no confusion will arise, we write $K\subset \Omega$ and consider it as a subset of $\Omega$.

As will be made clear in the sequel, a key ingredient in (our approach to) the Viterbo restriction map is the notion of: \emph{Floer cohomology of Hamiltonian in an action window}. The case of relevance to us is the following: given a smooth, autonomous, Hamiltonian $H$, and an action window $I$ whose endpoints are not contained in the spectrum of action values of 1-periodic orbits of $H$, there is a canonically defined vector space $\mathit{HF}_{I}(H)$. We develop the necessary aspects of this theory in \S\ref{sec:appendix-filtered}.

The table of contents of this section is as follows: \S\ref{sec:viterbo-restriction-system} discusses the special Hamiltonian system so that the Viterbo restriction map is represented as the relationship between a Floer complex (computing $\mathit{SH}_{c}(\Omega)$) and its quotient complex in a certain action window; \S\ref{sec:no-escape-lemma} discusses the well-known ``no-escape'' lemma which shows that the aforementioned quotient complex is isomorphic to $\mathit{SH}_{b}(K)$, and there we conclude the proof of Theorem \ref{theorem:VR-main}. In \S\ref{sec:truncated-viterbo-restriction} we discuss Zhou's truncated version (Theorem \ref{theorem:VR-trunc}).

\subsection{The Viterbo restriction system}
\label{sec:viterbo-restriction-system}

Denote by $r$ the radial coordinate for $\Omega$ and by $\rho$ the radial coordinate for $K$. For $b>c$, such that $b\not\in \mathit{Spec}(K)$ and $c\not\in \mathit{Spec}(\Omega)$, consider the non-smooth Hamiltonian:
\begin{equation}\label{eq:VR-system-non-smooth}
  H_{a,c}=\left\{
    \begin{aligned}
      &b\rho&&\text{ in }K\\
      &b&&\text{ in }\Omega-K\\
      &cr+b-c&&\text{ in }W-\Omega,
    \end{aligned}
  \right.
\end{equation}
where $W$ denotes the completion of $\Omega$. Because this system is not smooth, it does not have an obviously defined Hamiltonian Floer theory. Thus we first digress on how to discuss the Floer cohomology of $H_{b,c}$.

\subsubsection{Mollification formula}
\label{sec:moll-form}

For each number $\delta$ such that $K\subset \set{r<\delta}$, consider the function (see Figure \ref{fig:viterbo-restriction}):
\begin{equation}\label{eq:non-smooth-first-pass}
  H_{\delta,b,c}=
  \left\{
    \begin{aligned}
      &2\delta b&&\text{ on }\rho\le \delta,\\
      &b\rho+\delta b&&\text{ on }\rho\in [\delta,1-\delta],\\
      &b&&\text{ on }\rho\in [1-\delta,1]\text{ and in }\Omega-K,\\
      &cr+b-c&&\text{ on }r\ge 1.
    \end{aligned}
  \right.
\end{equation}

We will smooth this function using convolution with a mollifier.
\begin{definition}
  A \emph{mollifier} $f:\R\to [0,\infty)$ is
  a smooth function such that $\int f(x)dx=1$
  and so that $f$ is compactly supported in
  $(-1,1)$.
\end{definition}
Associated to such $f$, we denote the convolution by the
formula:
\begin{equation*}
  (fH)_{\delta,b,c}=f_{\delta}\ast H_{\delta,b,c}
\end{equation*}
where $f_{\delta}(x)=\delta^{-1}f(\delta
x)$. It is defined by the usual convolution
formula from real analysis, except that one
uses either $\rho$ or $r$ depending on the
region of the domain in
\eqref{eq:non-smooth-first-pass}. Because the
mollification only sees parts of the domain
within values $\pm \delta$ for the
coordinates $\rho$ and $r$, it is readily
checked that the function is smooth on all of
$\Omega$.

As $(fH)_{\delta,b,c}$ is a smooth function, it has a well-defined filtered Floer cohomologies $\mathit{HF}_{I}((fH)_{\delta,b,c})$ for every action window $I\subset \R$ that is disjoint from the spectrum of action values of $(fH)_{\delta,b,c}$, as explained in \S\ref{sec:appendix-filtered}.

\subsubsection{The action spectrum of the Viterbo restriction system}
\label{sec:acti-spectr-viterbo}

In this section we prove a result involving certain action windows for the smoothed Hamiltonian systems $(fH)_{\delta,b,c}$.

\begin{definition}\label{definition:admissible-tuple-for-VR}
  If $0<c<b$, and $b\not\in \mathit{Spec}(K)$ and $c\not\in \mathit{Spec}(\Omega)$, then we define the constant:
  \begin{equation*}
    \delta_{0}(b,c,K):=(7b)^{-1}\min\set{b-\sigma(b),b-c},
  \end{equation*}
  where $\sigma(b)$ is the greatest element in $\mathit{Spec}(K)$ less than $b$.

  Let us say that a tuple $(b,c,\delta,f)$ is \emph{admissible for the Viterbo restriction map} provided:
  \begin{enumerate}
  \item $0<b<c$,
  \item $b\not\in \mathit{Spec}(K)$, $c\not\in \mathit{Spec}(\Omega)$,
  \item $f$ is a mollifier, and,
  \item $\delta<\delta_{0}(b,c,K)$.
  \end{enumerate}
\end{definition}
\begin{lemma}\label{lemma:admissible-action-tau-2tau}
  Given any admissible tuple $(b,c,\delta,f)$ as in Definition \ref{definition:admissible-tuple-for-VR}, then the constant $\tau=\tau(b,c,\delta,f)=3b\delta$ satisfies that $(-\infty,0]\cup [\tau,2\tau]$ is disjoint from the spectrum of $(fH)_{\delta,b,c}$. In particular, for the interval $I=[0,\tau]$, the filtered Floer cohomology $\mathit{HF}_{I}((fH)_{\delta,b,c})$ is well-defined.
\end{lemma}
\begin{remark}
  The relevance of $[\tau,2\tau]$ is its appearance in our approach to the product structure on filtered Floer cohomology; see \S\ref{sec:conn-pair-pants}. The relevance of the numbers $7$ and $3$ is their use in the proof of Lemma \label{lemma:canonical-subquotient} below.
\end{remark}

\begin{proof}
  The key is to inspect the action values of the orbits of type \ref{item:T1}, \ref{item:T2}, and \ref{item:T3}, as shown in Figure \ref{fig:viterbo-restriction}. The action values can be estimated using the well-known ``Viterbo slope formula'' giving the action of an orbit of $H=h(r)$ or $H=h(\rho)$ in terms of the ``$y$-intercept'' of a certain tangent line to the graph of $h$. This yields:
  \begin{enumerate}[label=(T\arabic*)]
  \item\label{item:T1} actions in the interval $(b\delta,2b\delta]$,
  \item\label{item:T2} actions in the interval $(b-\sigma(b),b]$,
  \item\label{item:T3} actions in the interval $(b-c,b]$,
  \end{enumerate}
  If we pick $\delta$ smaller than $\delta_{0}(b,c,K)$ from the statement, then it follows that $7b\delta<\min\set{b-\sigma(b),b-c}$, and thus the set $(-\infty,0]\cup [\tau,2\tau]$ is indeed disjoint from the action spectrum of $(fH)_{\delta,b,c}$, as desired.
\end{proof}
\begin{remark}
  The proof of Lemma \ref{lemma:admissible-action-tau-2tau} also shows that the filtered Floer cohomology $\mathit{HF}_{I}((fH)_{\delta,b,c})$ is generated by those orbits of type \ref{item:T1} (in a rather loose sense, as the filtered Floer cohomology of \S\ref{sec:appendix-filtered} is not exactly defined in terms of generators but rather via an abstract limiting process). This observation is used in the discussion of the no-escape lemma in \S\ref{sec:no-escape-lemma}.
\end{remark}

\begin{figure}[h]
  \centering
  \begin{tikzpicture}[xscale=1.5]
    \draw (0,-0.8)--(0.1,-0.8)--(0.9,0)--(1,0);
    \draw[dotted] (1,0)--node[above]{$\Omega-K$}(2.4,0);
    \draw[dashed] (1,0.5)--+(0,-1.5)node[below]{$\rho=1$};
    \draw[dashed] (2.5,0.5)--+(0,-1.5)node[below]{$r=1$};
    \begin{scope}[shift={(2.5,0)}]
      \draw (-0.1,0)--(0,0)--(1,0.5);
    \end{scope}
    \path (0.1,-0.8)node[circle,draw,inner sep=8pt]{} -- (0.9,0)node[circle,draw,inner sep=8pt]{}--(2.5,0)node[circle,draw,inner sep=8pt]{};
    \path[shift={(-0.3,0.5)}] (0.1,-0.8)node{T1}--(0.9,0)node{T2}; \path[shift={(0.3,0.5)}](2.5,0)node {T3};
  \end{tikzpicture}
  \caption{The system of type $H_{\delta,b,c}$, when mollified, produces three types of orbits; the approximate positions of these orbits are shown in the circled regions, and are called type \ref{item:T1}, \ref{item:T2}, \ref{item:T3} (going from left to right). There are also constant orbits in the region $\Omega-K$, which are considered as type \ref{item:T3}.}
  \label{fig:viterbo-restriction}
\end{figure}
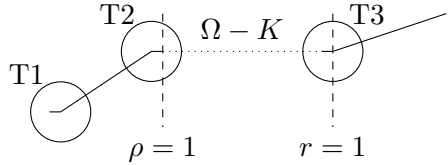

\subsubsection{Canonical subquotient map}
\label{sec:canon-subq-map}

Let us denote by $\mathscr{T}(b,c)$ the category of tuples $\theta=(\delta,f)$ so that $(b,c,\delta,f)$ are admissible in the sense of Definition \ref{definition:admissible-tuple-for-VR}. By the result of Lemma \ref{lemma:admissible-action-tau-2tau}, for each $\theta\in \mathscr{T}(b,c)$ one has a quotient map:
\begin{equation}\label{eq:canonical-subquotient-map}
  \mathfrak{Q}_{\theta}:\mathit{HF}((fH)_{\delta,b,c})\to \mathit{HF}_{I}((fH)_{\delta,b,c})\text{ where $I(\theta)=[0,\tau]$, $\tau=3b\delta$.}
\end{equation}
Technically we have two maps:
\begin{equation*}
  \mathit{HF}\to \mathit{HF}_{(-\infty,3b\delta]}\leftarrow \mathit{HF}_{[0,3b\delta]},
\end{equation*}
but because there are no orbits with action in $(-\infty,0]$ the second arrow is an isomorphism and can be inverted; see Remark \ref{remark:enlarging-intervals}.

For future notational convenience, let us denote the domain and codomain of $\mathfrak{Q}_{\theta}$ by $T_{\theta}$ and $Q_{\theta}$ (for ``total'' and ``quotient,'' respectively).

The goal in this section is to explain why this map is independent of the choice of auxiliary data $\delta,f$, in the sense that there is a canonical functor:
\begin{equation*}
  \theta\in \mathscr{T}(b,c)\to \mathfrak{Q}_{\theta}\in \text{(category of linear maps)}.
\end{equation*}
when $\mathscr{T}(b,c)$ is considered as an \emph{indiscrete groupoid} (i.e., a category with a unique morphism between any two objects). What this entails is constructing for each $\theta,\theta'$ maps $\mathfrak{c}_{\theta,\theta'}$ such that:
\begin{enumerate}[label=(I\arabic*)]
\item\label{item:I1} the diagram commutes:
  \begin{equation*}
    \begin{tikzcd}
      T_{\theta}\arrow[r,"\mathfrak{c}_{\theta,\theta'}"]\arrow[d,"\mathfrak{Q}_{\theta}",swap]&   T_{\theta'}\arrow[d,"\mathfrak{Q}_{\theta'}"]\\
      Q_{\theta}\arrow[r,"\mathfrak{c}_{\theta,\theta'}"]&Q_{\theta'}
    \end{tikzcd}
  \end{equation*}
  \item\label{item:I2} $c_{\theta,\theta'}=\id$ if $\theta=\theta'$,
  \item\label{item:I3} for all triples $\theta,\theta',\theta''$ one has $\mathfrak{c}_{\theta',\theta''}\circ \mathfrak{c}_{\theta,\theta'}$.
\end{enumerate}
For related discussion, see \cite[Theorem 5.2]{hofer-salamon-95}.

The construction of the maps $\mathfrak{c}_{\theta,\theta'}$ uses the simplest type of continuation map, namely those associated to the linear interpolation between $(fH)_{\delta,b,c}$ and $(f'H)_{\delta',b,c}$, as in \S\ref{sec:floer-cohom-smooth}.

\begin{lemma}\label{lemma:canonical-subquotient}
  The continuation maps $\mathfrak{c}_{\theta,\theta'}$ associated to linear interpolations (as above) for $\theta,\theta'\in \mathscr{T}(b,c)$ satisfies \ref{item:I1}, \ref{item:I2}, and \ref{item:I3}.
\end{lemma}
\begin{proof}
  Given two Hamiltonian systems $(fH)_{\delta,b,c}$ (input) and $(f'H)_{\delta',b,c}$ (output), one considers the linear interpolation:
  \begin{equation*}
    H_{s}:=\beta(s)(fH)_{\delta,b,c}+(1-\beta(s))(f'H)_{\delta',b,c}
  \end{equation*}
  for a cut-off function $\beta$ as in Definition \ref{definition:cut-off}, and defines the map by counting the solutions of:
  \begin{equation*}
    \partial_{s}u+J(u)(\bd_{t}u-X_{H_{s}}(u))=0.
  \end{equation*}
  By the well-known estimates of \cite[\S2.4]{schwarz-pacific-j-math-2000}, such a continuation map decreases action by at most the quantity:
  \begin{equation*}
    \max_{W}\left((fH)_{\delta,b,c}-(f'H)_{\delta',b,c}\right)\le 2b\delta.
  \end{equation*}
  In particular, if $\min\set{b-c-2b\delta,b-\sigma(b)-2b\delta}>2b\delta,$ then the continuation maps must preserve the subcomplex generated by \ref{item:T2} and \ref{item:T3}. This is satisfied since:
  \begin{equation*}
    4b\delta<\min\set{b-c,b-\sigma(b)}.
  \end{equation*}
  This proves that \ref{item:I1} holds. Property \ref{item:I2} is a standard fact, and is not affected the filtrations; see \cite[Theorem 4]{floer-comm-math-phys-1989}.

  Finally we comment on the functoriality property \ref{item:I3}. To compose two continuations maps, we use the homotopy argument which involves gluing together two continuation cylinders; during this process, the continuation cylinders can drop action by at most $4b\delta$ (the estimate is similar to the ones of \cite[\S2.4]{schwarz-pacific-j-math-2000} given above, and is made in more precise form in \S\ref{sec:floer-cohom-smooth}). Since:
  \begin{equation*}
    6b\delta<\min\set{b-c,b-\sigma(b)},
  \end{equation*}
  the usual chain homotopy map also preserves the subcomplexes generated by orbits of type \ref{item:T2} and \ref{item:T3}, and so the maps $\mathfrak{c}_{\theta',\theta''}\circ \mathfrak{c}_{\theta,\theta'}$ and $\mathfrak{c}_{\theta,\theta''}$ are chain homotopic and induce the same map on $Q_{\theta}$ (as well as $T_{\theta}$).
\end{proof}

\begin{lemma}
  Let $\mathfrak{Q}_{b,c}:T_{b,c}\to Q_{b,c}$ be the limits of the functor $\theta \mapsto \mathfrak{Q}_{\theta}$. Then there is a canonical isomorphism $\mathit{SH}_{c}(\Omega)\to T_{b,c}$.
\end{lemma}
\begin{proof}
  The group $\mathit{SH}_{c}(\Omega)$ is the limit of the larger functor (also defined on an indiscrete groupoid), which considers all Hamiltonian systems which agree with $cr$ outside of a compact set. Thus there is a map $\mathit{SH}_{c}(\Omega)\to T_{b,c}$ given by universal property of the limit, and because both functors are defined on indiscrete groupoids, this map is an isomorphism.
\end{proof}

Let us agree to also denote by $\mathfrak{Q}_{b,c}:\mathit{SH}_{c}(\Omega)\to \mathit{Q}_{b,c}$ the resulting map.

\subsubsection{Compatibility with additional structures}
\label{sec:comp-with-prod}

In this section we use the results of \S\ref{sec:comp-struct-1} to deduce that the maps $\mathfrak{Q}_{b,c}$ constructed in the previous section respect additional structures.

The easy statements to conclude (using \S\ref{sec:comp-struct-1}) are:
\begin{itemize}
\item $\mathfrak{Q}_{b,c}$ commutes with $\Delta$,
\item $\mathfrak{Q}_{b,c}$ commutes with $\mathit{PSS}$;
\end{itemize}
the statements involving the product and the continuation map are a bit more involved as they involving changing the slopes $b,c$.
\begin{lemma}\label{lemma:A-POP}
  Given positive real numbers $w_{i}$, $i=0,1$, and a pair $0<c<b$ such that:
  \begin{itemize}
  \item $cw_{0},cw_{1},c(w_{0}+w_{1})\not\in \mathit{Spec}(\Omega)$,
  \item $bw_{0},bw_{1},b(w_{0}+w_{1})\not\in \mathit{Spec}(K)$,
  \end{itemize}
  the pair of pants product $\mu$ induces a commutative diagram:
  \begin{equation*}
    \begin{tikzcd}
      \mathit{SH}_{cw_{0}}(\Omega)\otimes \mathit{SH}_{cw_{1}}(\Omega)
      \arrow[r,"\mu"]\arrow[d,"\mathfrak{Q}\otimes \mathfrak{Q}",swap]&\mathit{SH}_{c(w_{0}+w_{1})}(\Omega)\arrow[d,"\mathfrak{Q}"]\\
      Q_{bw_{0},cw_{0}}\otimes Q_{bw_{1},cw_{1}}\arrow[r,"\mu"]&Q_{b(w_{0}+w_{1}),c(w_{0}+w_{1})}
    \end{tikzcd}
  \end{equation*}
\end{lemma}
\begin{proof}
  This mostly follows from the arguments in \S\ref{sec:conn-pair-pants}; we explain how the results in that section are used. Briefly, one fixes $(\delta,f)$ so that
  \begin{itemize}
  \item $(\delta,f)\in \mathscr{T}(wb,wc)$ for $w=w_{0},w_{1},w_{0}+w_{1}$,
  \end{itemize}
  and picks connection of type $S$ on the pair of pants with $H=(fH)_{a,b}$ and weights $(w_{0},w_{1};w_{0}+w_{1})$. The framework of \S\ref{sec:conn-pair-pants} yields some operation $\mu$ making the diagram commute, but it a priori depends on the choice of $\delta$ and $f$. Similar arguments to the ones used in the proof of Lemma \ref{lemma:canonical-subquotient} then show the result is independent of $\delta$ and $f$, if $\delta$ is small enough so that $(\delta,f)\in \mathscr{T}(wb,wc)$ as above.
\end{proof}

\begin{lemma}
  Given positive real numbers $w\le w'$, and $0<c<b$ such that:
  \begin{itemize}
  \item $cw,cw'\not\in \mathit{Spec}(\Omega)$,
  \item $bw',bw'\not\in \mathit{Spec}(K)$,
  \end{itemize}
  the continuation map construction induces a commutative diagram:
  \begin{equation*}
    \begin{tikzcd}
      \mathit{SH}_{cw}(\Omega)\arrow[r,"\mathfrak{c}"]\arrow[d,"\mathfrak{Q}",swap]&\mathit{SH}_{cw'}(\Omega)\arrow[d,"\mathfrak{Q}"]\\
      Q_{bw,cw}\arrow[r,"\mathfrak{c}"]&Q_{bw',cw'}.
    \end{tikzcd}
  \end{equation*}
\end{lemma}
\begin{proof}
  The proof is similar to Lemma \ref{lemma:A-POP}: one fixes $\delta,f$ with $\delta$ small enough that:
  \begin{itemize}
  \item $(\delta,f)\in \mathscr{T}(wb,wc)$, and similarly for $w'$
  \end{itemize}
  then uses \S\ref{sec:connections-type-s} to define a map, and then proves the result is independent of the choice of $\delta,f$ via a similar argument to the one used in the proof of Lemma \ref{lemma:canonical-subquotient}.
\end{proof}

\subsection{No-escape lemma}
\label{sec:no-escape-lemma}

We use the no-escape lemma \cite[Lemma 7.2]{abouzaid-seidel-GT-2010}. This will be used to ensure that solutions to the Floer equations with all asymptotics in $K\subset \Omega$ remain localized to $K$. The algebraic consequence of this lemma is that the group $Q_{b,c}$ in \S\ref{sec:canon-subq-map} is canonically isomorphic to $\mathit{SH}_{b}(K)$. Composing this with the previously defined map $\mathfrak{Q}:\mathit{SH}_{c}(\Omega)\to Q_{b,c}$ yields a map $\mathfrak{R}:\mathit{SH}_{c}(\Omega)\to \mathit{SH}_{b}(K)$ which is the Viterbo restriction map satisfying Theorem \ref{theorem:VR-main}.

In the previous section \S\ref{sec:viterbo-restriction-system}, we were able to be rather agnostic about the choice of almost structure $J$; we only required that was $\omega$-tame and $Z$ invariant outside of some large compact set. In this section we will need to impose a futher condition defined in terms of the radial coordinate $\rho$ for $K$:
\begin{itemize}
\item $J$ satisfies $\lambda_{K}\circ J=\d \rho$ near the level set $\rho=1/2$.
\end{itemize}
\begin{lemma}\label{lemma:no-escape}
  Let $(\delta,f)\in \mathscr{T}(b,c)$, and suppose $\delta<1/2$. Then there is a canonical identification:
  \begin{equation*}
    \mathfrak{I}_{b,c}:Q_{b,c}\simeq \mathit{SH}_{b}(K),
  \end{equation*}
  commuting with the additional structures discussed in \S\ref{sec:comp-with-prod}.
\end{lemma}
\begin{proof}
  The key engine of the proof is \cite[Lemma 7.2]{abouzaid-seidel-GT-2010}. The statement ultimately follows from a chain level argument; for this reason, we will require a unpacking a bit the definition of $\mathit{HF}_{I}((fH)_{\delta,b,c})$. Since $(fH)_{\delta,b,c}$ is degenerate (in all likelihood), the recipe of \S\ref{sec:appendix-filtered} requires time dependent perturbation terms $D_{t}$; see, in particular, Definition \ref{definition:epsilon-admissible}.

  Let us select a special collection of perturbation terms $D_{t}$, namely those which are supported away from the hypersurface $\rho=1/2$. Then the orbits of $(fH)_{\delta,b,c}+D_{t}$ split into types \ref{item:T1}, \ref{item:T2}, and \ref{item:T3}, as in Figure \ref{fig:viterbo-restriction}.

  The smallness of $D_{t}$ is governed by which action window we want to use. Since $Q_{b,c}$ is defined in terms of the action window $I=[0,\tau]$ for $(fH)_{\delta,b,c}+D_{t}$ where $\tau=3b\delta$, we know that all of the orbits used to define:
  \begin{equation*}
    \mathit{CF}_{I}((fH)_{\delta,b,c}+D_{t})
  \end{equation*}
  are of type \ref{item:T1}. It then follows from \cite[Lemma 7.2]{abouzaid-seidel-GT-2010} that all Floer differential cylinders joining orbits of type \ref{item:T1} remain entirely in the set $\rho<1/2$. By simply forgetting the complement $\Omega-K$, we conclude that:
  \begin{equation*}
    \mathfrak{I}:\mathit{HF}_{I}((fH)_{\delta,b,c}+D_{t})\simeq \mathit{SH}_{b}(K).
  \end{equation*}
  This provides the desired identification, although it a priori depends on $D_{t}$ and the choice of $(f,\delta)$.

  Using a simple compactness argument together with \cite[Lemma 7.2]{abouzaid-seidel-GT-2010}, one concludes that the identification does not depend on $D_{t},f,\delta$. This step is a bit subtle, since the naive limit $\delta\to 0$ yields a non-smooth Hamiltonian \eqref{eq:VR-system-non-smooth}. The trick we adopt is: in order to prove the commutativity on homology level, we  use the fact that continuation maps are unchanged (on homology level) if we ``slow down'' the continuation, by replacing $\beta(s)$ by $\beta(\epsilon s)$, where $\epsilon>0$ is a small real number, whenever we need to do a linear interpolation. By such an argument, we can take the limit $\epsilon\to 0$ and conclude the failure of commutativity on homology level leads to a Floer differential cylinder for some intermediate solution joining two orbits of type \ref{item:T1} which intersects the level $\rho=1/2$. Then we apply \cite[Lemma 7.2]{abouzaid-seidel-GT-2010} to conclude the desired contradiction.

  A similar argument then proves the remaining commutativity with additional structures.
\end{proof}

\begin{proof}[Proof of Theorem \ref{theorem:VR-main}]
  The ingredients are all in place. The identification $\mathfrak{I}_{b,c}$ provides the remaining ingredient needed to define $$\mathfrak{R}=\mathfrak{I}_{b,c}\circ \mathfrak{Q}_{b,c}:\mathit{SH}_{c}(\Omega)\to \mathit{SH}_{b}(K);$$ this map satisfies the desired properties by Lemma \ref{lemma:no-escape}.
\end{proof}

\subsection{Truncated Viterbo restriction}
\label{sec:truncated-viterbo-restriction}

In \cite[Proposition 3.1]{zhou-JSG-2022}, a statement similar to Theorem \ref{theorem:VR-trunc} is proven; namely, that there is a Viterbo restriction map $\mathit{SH}_{b}(\Omega)\to \mathit{SH}(T^{*}L)$ provided that that the rationality constant $\rho$ of a Lagrangian embedding $L\subset \Omega$ is larger than $2b$.

The overall structure of the proof of Theorem \ref{theorem:VR-trunc} is exceeding similar to the proof of Theorem \ref{theorem:VR-main}; the main difference occurs in \S\ref{sec:acti-spectr-viterbo}. The argument introduced orbits of types \ref{item:T1}, \ref{item:T2}, and \ref{item:T3} (see Figure \ref{fig:viterbo-restriction}). The non-exactness of the embedding implies that action values of orbits of types \ref{item:T1} and \ref{item:T2} computed in $\Omega$ differs from the action values computed in $K$ by addition of the of $\beta=\lambda_{K}-\iota^{*}\lambda_{\Omega}$ over the orbit. To be precise:
\begin{itemize}
\item orbits of type \ref{item:T1} have action in $(b\delta,2b\delta]+\rho\Z$,
\item orbits of type \ref{item:T2} have action in $(b-\sigma(b),b]+\rho\Z$,
\item orbits of type \ref{item:T3} have action in $(b-c,b]$, as before.
\end{itemize}

The key ingredient for the proof to proceed as in the exact case is:
\begin{lemma}
  If $\rho>b+c$, then orbits $\gamma$ whose action is in the window $I=[-c,7b\delta]$ must be of type \ref{item:T1} and satisfy $\int \gamma^{*}\beta=0$.
\end{lemma}
\begin{proof}
  We need $b-\rho<-c$ and $2b\delta-\rho<-c$ (the latter holds assuming that $\delta<1/2$, which can assume without loss of generality).
\end{proof}
\begin{remark}
  The filtered subcomplex of orbits with actions in $[-c,\infty)$ computes $\mathit{SH}_{c}(\Omega)$, as can be proven by a standard continuation isomorphism as in \cite{zhou-JSG-2022}. The subquotient map:
  \begin{equation*}
    \mathit{SH}_{c}(\Omega)\simeq \mathit{HF}_{[-c,\infty)}((fH)_{\delta,b,c})\to \mathit{HF}_{[-c,7b\delta)}((fH)_{\delta,b,c})
  \end{equation*}
  is isomorphic to the truncated Viterbo restricted map. Since the are no orbits with action in $[-c,0]$ or $[3b\delta,7b\delta]$ we can replace the final interval by $[0,3b\delta]$.
\end{remark}
The rest of the proof proceeds as in the exact case. We comment only that the conclusion that all orbits in the action window $[0,3b\delta]$ have $\int \beta=0$ is used again to obtain the no-escape lemma, and to obtain the conclusion on the free homotopy classes given in Theorem \ref{theorem:VR-trunc}.

\section{Filtered Floer cohomology}
\label{sec:appendix-filtered}

In this appendix, we review the theory of the filtered Floer cohomology groups associated to a smooth autonomous Hamiltonian $H$. Throughout we focus on the case relevant to the body of the paper:

\begin{definition}\label{definition:admissible-in-liouville}
  An autonomous Hamiltonian $H$ on the completion of a Liouville domain $\Omega$ is said to be admissible for $\Omega$ provided it agrees with $br$ outside of a compact set. If, in addition, we say $b\not\in \mathit{Spec}(\Omega)$, we say that $H$ is non-discriminant.
\end{definition}

We also fix an $\omega$-tame almost complex structure $J$ that is Liouville invariant outside of a compact set (and will refer to such almost complex structures as admissible). To simplify the exposition, as in Remark \ref{remark:significant-canonical}, we fix $J$ rather than treat it as an auxiliary choice.

\subsection{Filtered Floer cohomology for smooth Hamiltonians}
\label{sec:floer-cohom-smooth}

Any admissible Hamiltonian $H$ as above has a well-defined
\emph{spectrum of action values}:
\begin{equation*}
  \mathit{Spec}(H):=\set{\int_{0}^{1} H(x(t))dt-x^{*}\lambda:x\text{ is a 1-periodic orbit of }X_{H}}.
\end{equation*}
This is a closed and nowhere dense subset of the real line.

\begin{definition}\label{definition:epsilon-admissible}
  An $\epsilon$-admissible perturbation of a non-discriminant and admissible $H$ (Definition \ref{definition:admissible-in-liouville}) is any time-dependent, smooth, and compactly supported Hamiltonian function $D_{t}$ such that:
  \begin{enumerate}[label=(A\arabic*)]
  \item\label{item:A1} $H+D_{t}$ generates an isotopy whose time-$1$ map is non-degenerate, and all moduli spaces of Floer cylinders for $H+D_{t}$, using $J$, are cut transversally,
  \item\label{item:A2} $\mathit{Spec}(H+D_{t})$ lies in the $\epsilon$ neighborhood of $\mathit{Spec}(H)$,
  \item\label{item:A3} $\max \abs{D_{t}}\le \epsilon$,
  \end{enumerate}
\end{definition}

Let us denote by $\mathscr{O}(\epsilon)$ the
indiscrete groupoid whose objects are choices
of $\epsilon$-admissible perturbation, that
is to say, there is a unique morphism between
any two objects.

\begin{lemma}\label{lemma:3epsilon}
  If the endpoints of $I$ do not contain any
  points in the $3\epsilon$ neighborhood of
  $\mathit{Spec}(H)$, then the assignment:
  \begin{equation*}
    D_{t}\in \mathscr{O}(\epsilon)\mapsto \mathit{HF}_{I}(H+D_{t},J),
  \end{equation*}
  is a functor, using the standard
  continuation maps\footnote{Standard continuation maps are defined by counting the solutions of infinitesimally small perturbations of the linear interpolation from $H+D_{t}$ to $H+D_{t}'$, as in \cite[\S2.4]{schwarz-pacific-j-math-2000}.} associated to affine
  interpolations. Here we use the definition:
  \begin{equation*}
    \mathit{HF}_I(H+D_{t},J)=H(\mathit{CF}_{A>\inf
      I}(H+D_{t},J)/\mathit{CF}_{A>\sup I}(H+D_{t},J)).
  \end{equation*}
\end{lemma}

\begin{remark}
  As the domain of the functor is an
  indiscrete groupoid, the limit of this
  functor is isomorphic to any of the
  outputs:
  \begin{equation*}
    \lim_{D_{t}\in \mathscr{O}(\epsilon)} \mathit{HF}_{I}(H+D_{t},J)\simeq \mathit{HF}_{I}(H+D_{t},J).
  \end{equation*}
  In particular, the natural maps induced by
  the inclusion
  $\mathscr{O}(\epsilon')\subset
  \mathscr{O}(\epsilon)$ for
  $\epsilon'<\epsilon$ induce isomorphisms
  between the limits.
\end{remark}
\begin{definition}\label{definition:FFcohomology}
  The \emph{filtered Floer cohomology} of a
  smooth Hamiltonian $H$ and an interval $I$
  whose endpoints are disjoint from the
  spectrum $\mathit{Spec}(H)$ (henceforth we
  say $I$ is \emph{admissible} for $H$) is
  defined as:
  \begin{equation*}
    \mathit{HF}_{I}(H):=\lim_{D_{t}\in \mathscr{O}(\epsilon)}\mathit{HF}_{I}(H+D_{t},J),
  \end{equation*}
  where, in order to the make the definition
  unambiguous, we fix $\epsilon$ to be the
  infimal number such that the $4\epsilon$
  neighborhood of $\partial I$ intersects
  $\mathit{Spec}(H)$; this ensures the
  diagram is well-defined by Lemma
  \ref{lemma:3epsilon}.
\end{definition}

\begin{proof}[Proof of Lemma \ref{lemma:3epsilon}]
  By \ref{item:A2}
  it holds that $\mathit{Spec}(H+D_{t})$ is
  disjoint from the $2\epsilon$ neighbourhood
  of $\partial I$. In particular, the
  quotient complex:
  \begin{equation*}
    \mathit{CF}_{A>\inf I}(H+D_{t},J)/\mathit{CF}_{A>\sup I}(H+D_{t},J)
  \end{equation*}
  is generated by orbits whose actions are
  $\epsilon$ far from $\partial I$.

  By standard estimates on the change in action due to continuation maps due to \cite[\S2.4]{schwarz-pacific-j-math-2000}, the continuation map relating $H+D_{t}$ and $H+D_{t}'$ preserves the subcomplexes $\set{A>\sup I}$ and $\set{A>\inf I}$; otherwise the continuation map would need to drop action by at least $4\epsilon$, which is not possible with the standard continuation maps since $\abs{D_{t}}$ and $\abs{D_{t}'}$ are uniformly bounded by $\epsilon$ (so the total possible drop in action is only $2\epsilon$ plus an arbitrarily small error due to pertubations).

  It remains to prove that the assignment is functorial, i.e., that the composition of continuation maps is equal to the continuation map. Consider the continuation map associated to the family of Hamiltonian vector fields:
  \begin{equation}\label{eq:interpolation-app}
    X_{H}+(1-\beta(s))X_{D_{t}''}+(1-\beta(s-R))\beta(s)X_{D'_{t}}+\beta(s-R)\beta(s)X_{D_{t}},
  \end{equation}
  which, for large $R\gg 0$, defines a map
  which equals the composition $\mathfrak{c}'\circ \mathfrak{c}$ of the chain
  maps from $D_{t}$ to $D_{t}'$ and then from
  $D_{t}'$ to $D_{t}''$ (by the standard
  argument), and for $R\le -1$ defines the
  map $\mathfrak{c}''$ from $D_{t}$ to $D_{t}''$.

  The count of the rigid solutions lying over
  the parametric moduli space of solutions,
  as $R$ varies over $[-1,\infty)$, defines
  the chain homotopy $\mathfrak{k}$ term:
  \begin{equation*}
    \mathfrak{c}'\circ \mathfrak{c}-\mathfrak{c}''=d\mathfrak{k}+\mathfrak{k}d.
  \end{equation*}
  Since $\mathfrak{k}$ is also defined by
  counting the solutions of Floer's
  continuation equation, one can estimate the
  maximum drop in action in terms of the
  curvature term:
  \begin{equation*}
    \abs{\int_{\R\times \R/\Z}\partial_{s}H_{s,t}ds dt}\le 3\epsilon;
  \end{equation*}
  the estimate uses that $\abs{D_{t}}<\epsilon$, etc.\footnote{There is an arbitrarily small additional error due to the fact we must add an infinitesimally small perturbation to the formula in \eqref{eq:interpolation-app} (the size can be taken to be as small as desired).} Thus $\mathfrak{k}$ must also preserve the subcomplexes $A>\sup I$ and $A>\inf I$, since $3\epsilon$ is still less than the minimum $4\epsilon$ needed to leave the subcomplexes (as discussed above).
\end{proof}

\begin{remark}\label{remark:enlarging-intervals}
  Let $I_{1}=[\alpha,\beta]$ and $I_{2}=[\beta,\gamma]$, with $\alpha<\beta<\gamma$, and suppose $\partial I_{1}\cup \partial I_{2}$ is disjoint from the spectrum of $H$. Then, for $\epsilon$ small enough, there is a functorial exact triangle sending $D_{t}\in \mathscr{O}(\epsilon)$ to:
  \begin{equation*}
    \dots \to \mathit{HF}_{I_{2}}(H+D_{t},J)\to \mathit{HF}_{I_{1}\cup I_{2}}(H+D_{t},J)\to \mathit{HF}_{I_{1}}(H+D_{t},J)\to \dots.
  \end{equation*}
  Inspection of the proof of Lemma \ref{lemma:3epsilon} shows that it suffices that the $3\epsilon$ neighborhood of $\set{\alpha,\beta,\gamma}$ is disjoint from $\mathit{Spec}(H)$ in order for the functor to be well-defined. In order to make the definition of this ``functorial exact triangle'' unambiguous, let us fix $\epsilon$ to be the smallest number such that $4\epsilon$ neighborhood of $\set{\alpha,\beta,\gamma}$ intersects the spectrum of $H$, as in Definition \ref{definition:FFcohomology}.

  Taking the limit of this functorial exact triangle gives an induced exact triangle relating the invariant of Definition \ref{definition:FFcohomology}:
  \begin{equation*}
    \dots \to \mathit{HF}_{I_{2}}(H)\to \mathit{HF}_{I_{1}\cup I_{2}}(H)\to \mathit{HF}_{I_{1}}(H)\to \dots.
  \end{equation*}
  
  If $I_{2}\cap \mathit{Spec}(H)=\emptyset$, then the restriction map:
  \begin{equation*}
    \mathit{HF}_{I_{1}\cup I_{2}}(H)\to \mathit{HF}_{I_{1}}(H)
  \end{equation*}
  is an isomorphism, and, if $I_{1}\cap \mathit{Spec}(H)=\emptyset$, then the ``inclusion'' map:
  \begin{equation*}
    \mathit{HF}_{I_{2}}(H)\to \mathit{HF}_{I_{1}\cup I_{2}}(H)
  \end{equation*}
  is an isomorphism.
\end{remark}

\subsection{Interfacing with the TQFT}
\label{sec:interf-with-tqft}

Recall that a \emph{Hamiltonian connection} is a convenient way to package the inhomogeneous term in Floer's equation; for a more detailed introduction introduction, we refer the reader to \cite[\S8]{mcduff-salamon-book-2012} and \cite[\S3]{brocic-cant-arXiv-2025}. Hamiltonian connections over a surface $\Sigma$ are a special type of Ehresmann connection on the fibre bundle $\Sigma\times W\to \Sigma$.

Hamiltonian connections $\mathfrak{H}$ are defined by \emph{connection potentials}, namely $1$-forms $\mathfrak{a}\in \Omega^{1}(\Sigma\times W)$ which vanish on the \emph{vertical distribution} (tangent to the fibres $\set{\mathit{pt}}\times W$). In local coordinates $s+it$ on $\Sigma$, such a potential can be expressed as:
\begin{equation*}
  \mathfrak{a}=K_{s,t}\d s+H_{s,t}\d t
\end{equation*}
where $K_{s,t},H_{s,t}$ are functions on $W\times \C$. The connection $\mathfrak{H}$ is defined to be the ``orthogonal complement'' of the vertical distribution with respect to the two-form $\omega-\d \mathfrak{a}$. More prosaically, $\mathfrak{H}$ is locally spanned by the vector fields:
\begin{equation*}
  \partial_{s}+X_{K_{s,t}}\quad\text{and}\quad\partial_{t}+X_{H_{s,t}}.
\end{equation*}
This concludes are recollection of Hamiltonian connections.

Let $\Sigma$ be a Riemann surface with $N+1$ cylindrical ends, decomposed into $N$ inputs and $1$ output and let $\mathfrak{H}$ be a Hamiltonian connection over $\Sigma$ whose connection potential $\mathfrak{a}$ is locally expressed as $\mathfrak{a}=K_{s,t}\d s+H_{s,t}\d t$ where $K_{s,t},H_{s,t}$ are admissible for all $s,t$ in the sense of Definition \ref{definition:admissible-in-liouville}; in the cylindrical ends we suppose $\mathfrak{a}$ is written as:
\begin{equation*}
  \mathfrak{a}=H_{k}dt\text{ for }k=1,\dots,N,\infty,
\end{equation*}
where $\infty$ corresponds to the unique output. We assume that all $H_{k}$ are admissible non-discriminant Hamiltonians as in \S\ref{sec:floer-cohom-smooth}.

Following the well-known recipe as in \cite[\S8]{mcduff-salamon-book-2012} and, say, \cite{alizadeh-atallah-cant-math-z-2025,brocic-cant-arXiv-2025}, for each admissible almost complex structure $J$ and generic perturbation $\mathfrak{p}$ of $\mathfrak{H}$, one obtains a linear map:
\begin{equation}\label{eq:general-tqft}
  \mathit{CF}(H_{1}+D_{1,t},J)\otimes\dots\otimes \mathit{CF}(H_{N}+D_{N,t},J)\to \mathit{CF}(H_{\infty}+D_{\infty,t},J).
\end{equation}
Here the perturbation term $\mathfrak{p}$ of
$\mathfrak{H}$ carries with it perturbations
$D_{k,t}\in \mathscr{O}(\epsilon)$, for
$k=1,\dots,N,\infty$. A variant of the construction considers smooth families of domains $\Sigma_{\tau}$ and connections $\mathfrak{H}_{\tau}$ and counts the rigid elements appearing in a parametric moduli space (this is how chain homotopy terms and the BV operator are defined).

We wish to understand how such operations
interact with action windows. The change in
action values is governed by the so-called
``curvature'' term associated to the
Hamiltonian connection.
\begin{definition}
  The curvature term of a Hamiltonian connection is:
  \begin{equation}
    \label{eq:error-term}
    \mathit{curv}(\mathfrak{H})=\sup_{u:\Sigma\to W}\int_{\R\times \R/\Z}u^{*}\mathfrak{r},
  \end{equation}
  where $\mathfrak{r}$ is the curvature two
  form of $\mathfrak{H}$; see \cite[pp.\,14]{alizadeh-atallah-cant-math-z-2025}.
\end{definition}
\begin{lemma}\label{lemma:product-error-term}
  The error term satisfies:
  \begin{equation*}
    A_{H_{\infty}}(\gamma_{\infty})\ge \sum_{k=1}^{N} A_{H_{k}}(\gamma_{k})-\mathit{curv}(\mathfrak{H}).
  \end{equation*}
  whenever $\gamma_{1},\dots,\gamma_{N},\gamma_{\infty}$ are the asymptotics of a finite energy solution to Floer's equation for $\mathfrak{H}$, for any admissible $J$.
\end{lemma}
\begin{proof}
  This is a standard result in the theory of
  Floer's equation for Hamiltonian
  connections.
\end{proof}

In the next subsections \S\ref{sec:pants-appendix}, \S\ref{sec:continuation-maps}, \S\ref{sec:BV-operator-appendix} we will analyze in more detail the operations needed in the body of the paper.

\subsubsection{The pair of pants product}
\label{sec:pants-appendix}

The interaction between action windows and the product structure is a bit subtle. The reason is that defining a binary operation with action windows $I_{1},I_{2},I_{\infty}$, say:
\begin{equation}\label{eq:product-map-windows}
  \mu:\mathit{HF}_{I_{1}}(H_{1})\otimes \mathit{HF}_{I_{2}}(H_{2})\to \mathit{HF}_{I_{\infty}}(H_{\infty}),
\end{equation}
typically requires that the ideal subcomplex:
\begin{equation}\label{eq:subcomplex-ideal}
  \mathit{CF}_{A>\sup I_{1}}\otimes \mathit{CF}_{A>\inf I_{2}} + \mathit{CF}_{A>\inf I_{1}}\otimes \mathit{CF}_{A>\sup I_{2}}
\end{equation}
is mapped to zero, rather than just the subcomplex $\mathit{CF}_{A>\sup I_{1}}\otimes \mathit{CF}_{A>\sup I_{2}}$.

Fix $\mathfrak{H}$
on the pair of pants surface
$\Sigma=\C P^{1}-\set{1,2,\infty}$
with standard cylindrical ends as in
\S\ref{sec:interf-with-tqft}.

\begin{lemma}\label{lemma:window-spacing}
  Suppose $I_{1}=[\alpha_{1},\alpha_{1}+\ell_{1}]$, $I_{2}=[\alpha_{2},\alpha_{2}+\ell_{2}]$, $I_{\infty}=[\alpha_{\infty},\beta_{\infty}]$ are admissible intervals for $H_{1},H_{2},H_{\infty}$, and:
  \begin{equation*}
    \alpha_{\infty}\le \alpha_{1}+\alpha_{2}-\mathit{curv}(\mathfrak{H})\quad\text{and}\quad\beta_{\infty}\le \alpha_{1}+\alpha_{2}+\ell_{1}+\ell_{2}-\mathit{curv}(\mathfrak{H}).
  \end{equation*}
  Suppose additionally that:
  \begin{itemize}
  \item $[\alpha_{1}+\ell_{1},\alpha_{1}+\ell_{1}+\ell_{2}]\cap \mathit{Spec}(H_{1})=\emptyset$, and,
  \item $[\alpha_{2}+\ell_{2},\alpha_{2}+\ell_{1}+\ell_{2}]\cap \mathit{Spec}(H_{2})=\emptyset$;
  \end{itemize}
  this latter assumption is related to \eqref{eq:subcomplex-ideal}. Then the product morphism \eqref{eq:product-map-windows} is well-defined. The morphism is unchanged if $\mathfrak{H}$ is homotoped through connections
\end{lemma}

\begin{proof}[Proof sketch]
  Let us denote by
  $\gamma_{1},\gamma_{2},\gamma_{\infty}$ the
  inputs/outputs of a solution
  $u$ to Floer's equation for
  $\mathfrak{H}$. Then consider the following
  statements:
  \begin{enumerate}[label=(\roman*)]
  \item\label{item:action-statement-2}
    $\alpha_{1}<A_{H_{1}}(\gamma_{1})$ and
    $\alpha_{2}<A_{H_{2}}(\gamma_{2})$;
  \item\label{item:action-statement-3} $\alpha_{1}+\ell_{1}<A_{H_{1}}(\gamma_{1})$ or $\alpha_{2}+\ell_{2}<A_{H_{2}}(\gamma_{2})$;
  \item\label{item:action-statement-1} $A_{H_{\infty}}(\gamma_{\infty})<\alpha_{1}+\alpha_{2}+\ell_{1}+\ell_{2}-\mathit{curv}(\mathfrak{H})$.
  \end{enumerate}
  Note that \ref{item:action-statement-3} consists of those pairs $\gamma_{1}\otimes \gamma_{2}$ lying in the subcomplex \eqref{eq:subcomplex-ideal}; these should be killed if we want the map to be well-defined.

  The statement of Lemma \ref{lemma:window-spacing}
  implies that \ref{item:action-statement-3} is
  equivalent to:
  \begin{enumerate}[label=(\roman*),resume]
  \item\label{item:action-statement-4} $\alpha_{1}+\ell_{1}+\ell_{2}<A_{H_{1}}(\gamma_{1})$ or $\alpha_{2}+\ell_{1}+\ell_{2}<A_{H_{2}}(\gamma_{2})$;
  \end{enumerate}
  thus, the combination of
  \ref{item:action-statement-2},
  \ref{item:action-statement-4}, and Lemma
  \ref{lemma:product-error-term} contradicts
  \ref{item:action-statement-1}. This proves
  there are no solutions whose asymptotics
  satisfy
  \ref{item:action-statement-2}--\ref{item:action-statement-1},
  thus the map vanishes on the ideal
  \eqref{eq:subcomplex-ideal}. In this
  fashion, one concludes that the product
  operation \eqref{eq:general-tqft} induces a
  well-defined map
  \eqref{eq:product-map-windows}.
\end{proof}

We will now explain this with more details and give a proper proof keeping track of the perturbation terms $\mathfrak{p}$ and using the formal definition of the invariant.

\begin{definition}\label{definition:e-adm-mathfrak-p}
  A perturbation $\mathfrak{p}$ is said to be
  $\epsilon$-admissible for $\mathfrak{H}$ provided that:
  \begin{enumerate}[label=(P\arabic*)]
  \item $\mathfrak{p}$ contains the data of a tuple $(D_{1,t},D_{2,t},D_{\infty,t})$ such that each $D_{k,t}$ lies in $\mathscr{O}(\epsilon)$ for the respective $H_{k}$;
  \item the perturbation $\mathfrak{p}$ changes the connection in the cylindrical ends by modifying the connection potential to:
    \begin{equation*}
      (H_{k}+\beta(\pm s)D_{k,t})\d t
    \end{equation*}
    where the sign depends on the type of cylindrical end;
  \item\label{item:P3} over a small disk (with coordinates $s+it$) in the complement of the cylindrical ends, $\mathfrak{p}$ perturbs the existing connection potential:
    \begin{equation*}
      \mathfrak{a}=K_{s,t}ds+H_{s,t}dt
    \end{equation*}
    by the addition of a compactly supported term $k_{s,t}ds +h_{s,t}dt$ with:
    \begin{equation*}
      \abs{\partial_{s}h_{s,t}}+\abs{\partial_{t}k_{s,t}}+\abs{\omega(X_{h},X_{K})}+\abs{\omega(X_{H},X_{k})}+\abs{\omega(X_{h},X_{k})}<\epsilon;
    \end{equation*}
  \item the moduli spaces of pairs-of-pants for the perturbed data $(\mathfrak{H}_{\mathfrak{p}},J)$ are cut transversally.
  \end{enumerate}
The collection of all $\epsilon$-admissible perturbations is denoted $\mathscr{P}(\epsilon)$ (the dependence on the background connection $\mathfrak{H}$ and asymptotics $H_{k}$ is suppressed from the notation).
\end{definition}
\begin{lemma}\label{lemma:functorial-pop}
  Let $I_{1},I_{2},I_{\infty}$ be as in Lemma \ref{lemma:window-spacing}. If
  \begin{itemize}
  \item the $5\epsilon$ neighborhood of:
    \begin{equation*}
      \set{\alpha_{i}}\cup [\alpha_{i}+\ell_{i},\alpha_{i}+\ell_{1}+\ell_{2}]
    \end{equation*}
    is disjoint from $\mathit{Spec}(H_{i})$, and 
  \item the $3\epsilon$ neighborhood of $\partial I_{\infty}$ is disjoint from $\mathit{Spec}(H_{\infty})$,
  \end{itemize}
  then the map which associates to $\mathfrak{p}\in \mathscr{P}(\epsilon)$ the bilinear product:
  \begin{equation*}
    \mu_{\mathfrak{p}}:\mathit{HF}_{I_{1}}(H_{1}+D_{1,t})\otimes \mathit{HF}_{I_{2}}(H_{2}+D_{2,t})\to \mathit{HF}_{I_{\infty}}(H_{\infty}+D_{\infty,t})
  \end{equation*}
  is functorial when $\mathscr{P}(\epsilon)$ is considered as an indiscrete groupoid (using the standard continuation maps as in \S\ref{sec:floer-cohom-smooth}).
\end{lemma}
\begin{proof}
  Each map $\mu_{\mathfrak{p}}$ has a maximum
  drop in action of
  $\mathit{curv}(\mathfrak{H})+4\epsilon$. This
  implies that the ideal subcomplex
  \eqref{eq:subcomplex-ideal} is indeed mapped to zero. Indeed, any generator $\gamma_{1}\otimes \gamma_{2}$ from this subcomplex satisfies:
  \begin{itemize}
  \item $A(\gamma_{1})>\alpha_{1}+4\epsilon$,
  \item $A(\gamma_{2})>\alpha_{2}+\ell_{1}+\ell_{2}+4\epsilon$,
  \end{itemize}
  or the analogous inequalities with the roles of $1,2$ reversed; hence any output in $\mu_{\mathfrak{p}}(\gamma_{1}\otimes \gamma_{2})$ would necessarily satisfy:
  \begin{itemize}
  \item $A(\gamma_{\infty})>\alpha_{1}+\alpha_{2}+\ell_{1}+\ell_{2}+8\epsilon-4\epsilon-\mathit{curv}(\mathfrak{H})\ge \sup I_{\infty}+4\epsilon$,
  \end{itemize}
  and so $\gamma_{\infty}$ would be discarded in the count (since it is above $\sup I_{\infty}$).

  This proves the bilinear operation $\mu_{\mathfrak{p}}$ is well-defined. We will now explain why the operation is functorial. This amounts to proving the following square commutes:
  \begin{equation*}
    \begin{tikzcd}
      \mathit{HF}_{I_{1}}(H_{1}+D_{1,t})\otimes \mathit{HF}_{I_{2}}(H_{2}+D_{2,t})\arrow[d,"\mathfrak{c}\otimes \mathfrak{c}"]\arrow[r,"\mu_{\mathfrak{p}}"]& \mathit{HF}_{I_{\infty}}(H_{\infty}+D_{\infty,t})\arrow[d,"\mathfrak{c}"]\\
      \mathit{HF}_{I_{1}}(H_{1}+D'_{1,t})\otimes \mathit{HF}_{I_{2}}(H_{2}+D'_{2,t})\arrow[r,"\mu_{\mathfrak{p}'}"]& \mathit{HF}_{I_{\infty}}(H_{\infty}+D'_{\infty,t})
    \end{tikzcd}
  \end{equation*}
  In the usual argument proving this square commutes, one glues on continuation cylinders at all three ends; during this process, the maximum curvature one sees is $7\epsilon$ (roughly speaking, there are ``two'' continuation cylinders in each end, making $6\epsilon$, and the curvature introduced by behaviour of $\mathfrak{p}$ on the small disk contributes $\epsilon$). Since $7\epsilon$ is still smaller than the minimum gap $8\epsilon$, the chain homotopy terms will also vanish on the ideal subcomplex \eqref{eq:subcomplex-ideal}.
\end{proof}

To conclude this section, we explain how Lemma \ref{lemma:functorial-pop} induces a canonical operation \eqref{eq:product-map-windows} when the hypotheses of Lemma \ref{lemma:window-spacing} are satisfied.

\begin{definition}\label{definition:how-to-define-the-map}
  Assume the set-up of Lemma \ref{lemma:window-spacing}. Pick $\epsilon>0$ small
  enough that Lemma
  \ref{lemma:functorial-pop} applies. Given any $\mathfrak{p}\in \mathscr{P}(\epsilon)$, one defines the operation $\mu$ as follows:
  \begin{equation*}
    \begin{tikzcd}
      \mathit{HF}_{I_{1}}(H_{1})\otimes \mathit{HF}_{I_{2}}(H_{2})\arrow[d,"\simeq"]\arrow[r,"\mu"]&\mathit{HF}_{I_{\infty}}(H_{\infty})\arrow[d,"\simeq"]\\
      \mathit{HF}_{I_{1}}(H_{1}+D_{1,t})\otimes \mathit{HF}_{I_{2}}(H_{2}+D_{2,t})\arrow[r,"\mu_{\mathfrak{p}}"]& \mathit{HF}_{I_{\infty}}(H_{\infty}+D_{\infty,t}).
    \end{tikzcd}
  \end{equation*}
  That this is independent of $\mathfrak{p}$
  is a consequence of Lemma
  \ref{lemma:functorial-pop}. That this is
  independent of $\epsilon$ is tautological
  from its definition.
\end{definition}

\subsubsection{Continuation maps}
\label{sec:continuation-maps}

A continuation map depends on a Hamiltonian connection $\mathfrak{H}$ on a cylinder. The resulting operation is a map:
\begin{equation}\label{eq:continuation-map}
  \mathfrak{c}_{\mathfrak{H}}:\mathit{HF}_{I_{0}}(H_{0})\to \mathit{HF}_{I_{1}}(H_{1})
\end{equation}
for suitable action windows $I_{i}$, $i=0,1$. The main structural theorem (and analogue of Lemma \ref{lemma:window-spacing}) is:
\begin{lemma}\label{lemma:window-spacing-cont}
  Let $\mathfrak{H}$ be a connection from $H_{0}$ to $H_{1}$ as above. Suppose that $I_{i}=[\alpha_{i},\beta_{i}]$ is an admissible interval for $H_{i}$, $i=0,1$, and:
  \begin{equation}\label{eq:admissible-for-continuation}
    \alpha_{1}\le \alpha_{0}-\mathit{curv}(\mathfrak{H})\quad\text{and}\quad\beta_{1}\le \beta_{0}-\mathit{curv}(\mathfrak{H}).
  \end{equation}
  Then the continuation maps $\mathfrak{c}_{\mathfrak{H}}$ as in \eqref{eq:continuation-map} are well-defined (as a formal limit of the continuation maps associated to suitable perturbations of $\mathfrak{H}$).

  The map is invariant through homotopies of $\mathfrak{H}$ in the space of such connections, relative the ends of the cylinder, provided that \eqref{eq:admissible-for-continuation} holds for throughout the homotopy. Moreover, the map commutes with the ``enlarging interval'' maps of Remark \ref{remark:enlarging-intervals}.
\end{lemma}
\begin{proof}
  The proof is analogous to the proof of Lemma \ref{lemma:window-spacing}, and involves suitable versions of Definition \ref{definition:e-adm-mathfrak-p} and Lemma \ref{lemma:functorial-pop}. The remaining properties are either tautological from the construction, or follow standard arguments.
\end{proof}

\begin{remark}\label{remark:functorial}
  Given $(H_{0},H_{1},\mathfrak{H})$ and $(H_{1},H_{2},\mathfrak{H}')$, and intervals $I_{0},I_{1},I_{2}$ satisfying \eqref{eq:admissible-for-continuation}, then one can ``glue'' together $\mathfrak{H}'\# \mathfrak{H}$ to produce a new connection (the gluing depends on a choice of sufficiently large ``length'' but the resulting connection is well-defined up to homotopy relative ends) with:
  \begin{equation*}
    \mathit{curv}(\mathfrak{H}'\# \mathfrak{H})\le \mathit{curv}(\mathfrak{H}')+\mathit{curv}(\mathfrak{H}).
  \end{equation*}
  Standard arguments yield:
  \begin{equation*}
    \mathfrak{c}_{\mathfrak{H}'\# \mathfrak{H}}=\mathfrak{c}_{\mathfrak{H}'}\circ \mathfrak{c}_{\mathfrak{H}}:\mathit{HF}_{I_{0}}(H_{0})\to \mathit{HF}_{I_{2}}(H_{2}).
  \end{equation*}
  Another standard fact is that, if:
  \begin{itemize}
  \item $H_{0}=H_{1}=:H$, and
  \item $\mathfrak{H}$ has connection potential $\mathfrak{a}=H\d t$, so $\mathit{curv}(\mathfrak{H})=0$, and,
  \item we fix $I_{0}=I_{1}=:I$ to be any admissible interval for $H$,
  \end{itemize}
  then $\mathfrak{c}_{\mathfrak{H}}=\id$.

  These two facts assert that the assignment of $(H,I)$ to $\mathit{HF}_{I}(H)$ and $\mathfrak{H}$ to the continuation map $\mathfrak{c}_{\mathfrak{H}}$ defines a functor on a suitable category of pairs $(H,I)$, where morphisms are homotopy classes of connections $\mathfrak{H}$ satisfying \eqref{eq:admissible-for-continuation}. This general fact is used in the body of the paper to construct persistence modules.
\end{remark}

\subsubsection{BV operator on filtered Floer cohomology}
\label{sec:BV-operator-appendix}

Let $H$ be an admissible and non-discriminant Hamiltonian $H$ as in Definition \ref{definition:admissible-in-liouville} and let $I$ be an admissible interval as in Definition \ref{definition:FFcohomology}. In this section we explain how to construct the BV operator on the filtered Floer cohomology:
\begin{equation}\label{eq:BV-filtered-appendix}
  \Delta:\mathit{HF}_{I}(H)\to \mathit{HF}_{I}(H).
\end{equation}
We remark that this ``strict'' filtration (i.e., that we can take the same interval $I$ on both sides) uses the fact that $H$ is autonomous.

As in Definition \ref{definition:epsilon-admissible}, let $D_{t}\in \mathscr{O}(\epsilon)$ be an $\epsilon$-admissible perturbation. The BV operator is defined by counting the rigid solutions in the parametric moduli space of pairs $(\tau,u)$ such that:
\begin{equation}\label{eq:BV-equation}
  \left\{
    \begin{aligned}
      &u:\R\times \R/\Z\to W,\quad \tau\in \R/\Z,\\
      &X_{\tau,s,t}=X_{H}+(1-\beta(s))X_{D_{t-\tau}}+\beta(s)X_{D_{t}}+\text{perturbation term},\\
      &\bd_{s}u+J(u)(\bd_{t}u-X_{\tau,s,t}(u))=0,\\
      &\lim_{s\to-\infty}u(s,t)=\gamma_{-}(t-\tau),\quad\lim_{s\to+\infty}u(s,t)=\gamma_{+}(t).
    \end{aligned}
  \right.
\end{equation}
The perturbation term should be very small (adding curvature at most $\epsilon$). Such a solution is interpreted as contributing to the coefficient $\gamma_{+}\mapsto \gamma_{-}$. Counting the rigid elements defines a chain map:
\begin{equation*}
  \Delta:\mathit{CF}(H+D_{t})\to \mathit{CF}(H+D_{t}).
\end{equation*}

The interpolation between $X_{D_{t}}$ at the input and $X_{D_{t-\tau}}$ at the output introduces some small amount of curvature, and so some care is needed to ensure the map on filtered homology is well-defined as in \eqref{eq:BV-filtered-appendix}.

\begin{lemma}\label{lemma:BV-5epsilon}
  If the endpoints of $I$ do not contain any
  points in the $5\epsilon$ neighborhood of
  $\mathit{Spec}(H)$, then the assignment:
  \begin{equation*}
    D_{t}\in \mathscr{O}(\epsilon)\mapsto (\Delta:\mathit{HF}_{I}(H+D_{t},J)\to \mathit{HF}_{I}(H+D_{t},J)),
  \end{equation*}
  is well-defined and is a natural transformation.
\end{lemma}
\begin{proof}
  The proof of well-definedness is analogous to the proof of Lemma \ref{lemma:3epsilon}: even though we have introduced some curvature, the total drop in action is at most $3\epsilon$, by the estimates of \cite{schwarz-pacific-j-math-2000} applied to \eqref{eq:BV-equation}, which is larger than the buffers of width $8\epsilon$ around $\partial I$.

  To prove that the transformation is natural, one needs to check the commutativity of the following diagram:
  \begin{equation*}
    \begin{tikzcd}
      \mathit{HF}_{I}(H+D_{t})\arrow[d]\arrow[r,"\Delta"]&\mathit{HF}_{I}(H+D_{t})\arrow[d]\\
      \mathit{HF}_{I}(H+D_{t}')\arrow[r,"\Delta"]&\mathit{HF}_{I}(H+D_{t}')
    \end{tikzcd}
  \end{equation*}
  The usual argument proving this commutes
  involves a deformation of the equation
  \eqref{eq:BV-equation} through Hamiltonian
  connections with curvatures at most
  $7\epsilon$ (roughly, $2\epsilon$ for each
  vertical morphism and $3\epsilon$ for the
  horizontal morphism). Since $7\epsilon$ is
  still smaller than the gap $8\epsilon$, we
  conclude the desired result.
\end{proof}

This natural tranformation induces a well-defined map on the limits $\mathit{HF}_{I}(H)$, following similar lines to Definition \ref{definition:how-to-define-the-map}.

\subsubsection{PSS elements}
\label{sec:pss-elements}

We discuss the construction of \cite{piunikhin-salamon-schwarz-1996,frauenfelder-schlenk-IJM-2007} and its relation to the filtered Floer cohomology groups $\mathit{HF}_{I}(H)$ for particular action windows $I$.

Following \cite{brocic-cant-arXiv-2025}, let us fix a cycle $f:C\to W$ (a proper map) representing a cohomology class; e.g., the inclusion of a cotangent fiber into $W=T^{*}M$ represents a degree $n$ cohomology class. We also fix a Hamiltonian connection $\mathfrak{H}$ (with connection potential $\mathfrak{a}$) as in the start of \S\ref{sec:interf-with-tqft} on the cylinder with coordinates $s+it$ satisfying:
\begin{itemize}
\item $\mathfrak{a}=0$ for $s\ge s_{0}$,
\item $\mathfrak{a}=H\d t$ for $s\le -s_{0}$.
\end{itemize}
From this data, the goal is to construct an element $\mathit{PSS}_{\mathfrak{H}}(f)\in \mathit{HF}_{I}(H)$ for a suitable action window $I$ depending on the connection $\mathfrak{H}$, and show that the resulting element is independent of the choice of $f$ up to proper cobordism.

We recall that a solution of the PSS equation for $(\mathfrak{H},J,f)$ is a pair $(x,u)$ such that $u$ solves Floer's equation for $(\mathfrak{H},J)$ and such that $\lim_{s\to\infty}u(s,t)=f(x)$. Such solutions have a well-defined Fredholm theory when one compactifies the positive end as a removable singularity via the map $\R\times \R/\Z\to \C$ sending $s+it$ to $e^{-2\pi(s+it)}$. We have the analogue of Definition \ref{definition:e-adm-mathfrak-p}.

\begin{definition}\label{definition:e-adm-PSS-p}
  A perturbation $\mathfrak{p}$ is said to be $\epsilon$-admissible for $\mathfrak{H}$ provided:
  \begin{itemize}
  \item $\mathfrak{p}$ contains the data of $D_{t}\in \mathscr{O}(\epsilon)$ for $H$;
  \item the perturbation $\mathfrak{p}$ changes the connection in the output cylindrical end by modifying the connection potential to:
    \begin{equation*}
      (H+\beta(-s_{0}-s)D_{t})\d t;
    \end{equation*}
  \item over a small disk (with coordinates $s+it$) in the region $s>s_{0}$, $\mathfrak{p}$ perturbs the existing connection potential $\mathfrak{a}=0$ by the addition of a compactly supported term $k_{s,t}ds +h_{s,t}dt$ with:
    \begin{equation*}
      \abs{\partial_{s}h_{s,t}}+\abs{\partial_{t}k_{s,t}}+\abs{\omega(X_{h},X_{k})}<\epsilon;
    \end{equation*}
  \item the moduli spaces of PSS cylinders for the perturbed data $(\mathfrak{H}_{\mathfrak{p}},J,f)$ are cut transversally.
  \end{itemize}
  The collection of all $\epsilon$-admissible perturbations is denoted $\mathscr{S}(\epsilon)$.
\end{definition}

\begin{lemma}\label{lemma:functorial-pss}
  Let $(\mathfrak{H},f)$ be as above, and let $I=[\alpha,\beta]$ satisfy:
  \begin{itemize}
  \item $\alpha\le -\mathit{curv}(\mathfrak{H})$, and,
  \item the $4\epsilon$ neighborhood of $\partial I$ is disjoint from $\mathit{Spec}(H)$,
  \end{itemize}
  then the map which associates to $\mathfrak{p}\in \mathscr{S}(\epsilon)$ the chain:
  \begin{equation*}
    \mathit{PSS}(\mathfrak{H},f,J,\mathfrak{p})\in \mathit{HF}_{I}(H+D_{t})
  \end{equation*}
  obtained by counting the rigid solutions of the PSS equation for $(\mathfrak{H},f,J,\mathfrak{p})$ is functorial when $\mathscr{S}(\epsilon)$ is considered as an indiscrete groupoid (using the standard continuation maps as in \S\ref{sec:floer-cohom-smooth}).

  The resulting PSS elements are independent of $f$ up to proper cobordism and of $\mathfrak{H}$ up to homotopies through connections satisfying $\alpha\le -\mathit{curv}(\mathfrak{H})$.
\end{lemma}
\begin{proof}
  Since the perturbed connection $\mathfrak{H}$ still has $\mathfrak{a}=0$ at the input asymptotic, the energy of any PSS cylinder $u$ with output asymptotic $\gamma$ equals:
  \begin{equation*}
    0\le E(u)\le A_{H+D_{t}}(\gamma)+\mathit{curv}(\mathfrak{H})+2\epsilon
  \end{equation*}
  Thus it follows that:
  \begin{equation*}
    \alpha-2\epsilon\le -\mathit{curv}(\mathfrak{H})-2\epsilon \le A_{H+D_{t}}(\gamma).
  \end{equation*}
  The interval $[\alpha-2\epsilon,\alpha]$ is disjoint from $\mathit{Spec}(H+D_{t})$ because $D_{t}\in \mathscr{O}(\epsilon)$, and thus $\alpha\le A_{H+D_{t}}(\gamma).$ It therefore follows that the chain is well-defined in the complex $\mathit{CF}_{I}(H+D_{t})$.

  The operation being functorial means that, for $\mathfrak{p},\mathfrak{p}'$, the PSS elements are preserved by the standard continuation maps associated to the linear interpolation $H+D_{t}$ to $H+D_{t}'$. This follows similar lines to arguments in Lemma \ref{lemma:3epsilon}; the maximum curvature arising in homotopy used to prove:
  \begin{equation*}
    \mathfrak{c}(\mathit{PSS}(\mathfrak{H},f,J,\mathfrak{p}))=\mathit{PSS}(\mathfrak{H},f,J,\mathfrak{p}')
  \end{equation*}
  is at most $\mathit{curv}(\mathfrak{H})+3\epsilon$, thus the outputs $\gamma$ appearing in the chain homotopy term satisfy $\alpha-3\epsilon\le A_{H+D_{t}'}(\gamma)$, but (as above) $[\alpha-3\epsilon,\alpha]$ is disjoint from the spectrum so the chain homotopy takes values in $\mathit{CF}_{I}(H+D_{t}')$.

  By the same argument as in Definition \ref{definition:how-to-define-the-map}, taking a formal limit of the functor:
  \begin{equation*}
    \mathfrak{p}\mapsto (\mathit{PSS}(\mathfrak{H},f,J,\mathfrak{p})\in \mathit{HF}_{I}(H+D_{t}))
  \end{equation*}
  produces a well-defined element in $\mathit{PSS}(\mathfrak{H},f,J)\in \mathit{HF}_{I}(H)$. The final step that this does not depend on $f$ up to proper cobordism or $\mathfrak{H}$ up to homotopies (as in the statement) is a standard homotopy argument, and follows similar lines to the other arguments already given.
\end{proof}

\subsection{Compatibility of the structures}
\label{sec:comp-struct-1}

In this section we discuss the compatibility of the structures from \S\ref{sec:interf-with-tqft} with each other.

\subsubsection{Rotationally symmetric continuations and the BV operator}
\label{sec:rotationally-symmetric}

Suppose that $\mathfrak{H}$ is a Hamiltonian connection from $H_{0}$ to $H_{1}$ and the intervals $I_{0},I_{1}$ are admissible so that $\mathfrak{c}_{\mathfrak{H}}:\mathit{HF}_{I_{0}}(H_{0})\to \mathit{HF}_{I_{1}}(H_{1})$ is as in Lemma \ref{lemma:window-spacing-cont}. Suppose additionally that $\mathfrak{H}$ is rotationally symmetric in the sense that its connection potential $\mathfrak{a}$ satisfies:
\begin{equation*}
  \mathfrak{a}=K_{s}\d s+H_{s}\d t,
\end{equation*}
i.e., there is no $t$-dependence. The standard example of such a connection considered in the body of the text is:
\begin{equation*}
  K_{s}=0\quad\text{and}\quad H_{s}=\beta(s)H_{0}+(1-\beta(s))H_{1}.
\end{equation*}

\begin{lemma}
  In the above context, it holds that:
  \begin{equation*}
    \mathfrak{c}_{\mathfrak{H}}\circ \Delta=\Delta\circ \mathfrak{c}_{\mathfrak{H}},
  \end{equation*}
  as maps $\mathit{HF}_{I_{0}}(H_{0})\to \mathit{HF}_{I_{1}}(H_{1})$.
\end{lemma}
\begin{proof}
  The idea of the proof is that both sides are defined by counting the solutions of small perturbations of PDE:
  \begin{equation}\label{eq:BV-perturb}
    \left\{
      \begin{aligned}
        &u:\R\times \R/\Z\to W,\quad \tau\in \R/\Z,\\
        &\bd_{s}u-X_{K_{s}}(u)+J(u)(\bd_{t}u-X_{H_{s}}(u))=0,\\
        &\lim_{s\to-\infty}u(s,t)=\gamma_{-}(t-\tau),\quad\lim_{s\to+\infty}u(s,t)=\gamma_{+}(t).
      \end{aligned}
    \right.
  \end{equation}
  The recipe for perturbing it depends on whether one is computing $\Delta\mathfrak{c}_{\mathfrak{H}}$ or $\mathfrak{c}_{\mathfrak{H}}\Delta$, but, regardless, the two counts will be related by a chain homotopy term; during the usual homotopy argument the amount of curvature is controlled by the curvature of $\mathfrak{H}$ (which the intervals $I_{0},I_{1}$ already account for) plus small terms of the order $\epsilon$ due to the pertubations $D_{i,t}\in \mathscr{O}(\epsilon)$, etc. This observation is the key needed to obtain the desired result. The details are left to the reader.
\end{proof}

\subsubsection{Connections on the pair-of-pants via one-forms}
\label{sec:conn-pair-pants}

In this section we explain a mechanism for constructing connections on the pair-of-pants using real-valued one-forms. The idea appears in the work of \cite{seidel-IP-2008,ritter-jtopol-2013}. The resulting product operations interact nicely with certain continuation maps.

\begin{definition}\label{definition:split}
  A connection potential $\mathfrak{a}$ on $\Sigma$ is called \emph{split} if:
  \begin{equation*}
    \mathfrak{a}=H\mathfrak{b}
  \end{equation*}
  where $\mathfrak{b}\in \Omega^{1}(\Sigma)$ and $H\in C^{\infty}(W)$. If $\mathfrak{H}$ is generated by a split $\mathfrak{a}$, we say $\mathfrak{H}$ is split.
\end{definition}
A useful lemma about this is that the curvature term for such a connection can be computed explicitly:
\begin{lemma}\label{lemma:curvature-split}
  The curvature two-form $\mathfrak{r}$ of $\mathfrak{H}$ generated by a split connection potential $\mathfrak{a}=H\mathfrak{b}$ is:
  \begin{equation*}
    \mathfrak{r}=H\d \mathfrak{b}.
  \end{equation*}
  In particular, if $H\ge 0$ and $\d \mathfrak{b}\le 0$ (with respect to the orientation of $\Sigma$), then $\mathit{curv}(\mathfrak{H})=0$ (i.e., the operations defined using $\mathfrak{H}$ are filtered).
\end{lemma}
\begin{proof}
  Locally write $\mathfrak{b}=k_{s,t}\d s+h_{s,t}\d t$. Compute:
  \begin{equation*}
    \mathfrak{r}=H(\partial_{s}h_{s,t}-\partial_{t}k_{s,t})\d s\wedge\d t+k_{s,t}h_{s,t}\omega(X_{H},X_{H})=H\d \mathfrak{b},
  \end{equation*}
  as desired.
\end{proof}

\begin{definition}\label{definition:type-S-pop}
  A split connection $\mathfrak{H}$ with $\mathfrak{a}=H\mathfrak{b}$ and $H\ge 0$ and $\d \mathfrak{b}\le 0$ as in Lemma \ref{lemma:curvature-split} on $\Sigma=\C P^{1}-\set{0,1,\infty}$ is said to be of \emph{type S with weights} $(w_{0},w_{1};w_{\infty})$ provided:
  \begin{itemize}
  \item in standard cylindrical ends $s+it$ around $0,1,\infty$, $\mathfrak{b}$ appears as $w_{i}\d t$ for $i=0,1,\infty$, with $w_{\infty}\ge w_{0}+w_{1}$,
  \end{itemize}  
  As in \S\ref{sec:pants-appendix}, the ends around $0,1$ are oriented as positive ends (inputs) and the end around $\infty$ is oriented as a negative end (outputs). Note that the condition $w_{\infty}\ge w_{0}+w_{1}$ necessarily follows from the requirement that $\d \mathfrak{b}\le 0$.
\end{definition}

Similarly we define continuation cylinders of type S:
\begin{definition}\label{definition:type-S-cyl}
  A split connection $\mathfrak{H}$ with $\mathfrak{a}=H\mathfrak{b}$ and $H\ge 0$ and $\d \mathfrak{b}\le 0$ as in Lemma \ref{lemma:curvature-split} on $\Sigma=\R\times \R/\Z$ is said to be of \emph{type S with weights} $(w_{+};w_{-})$ provided $\mathfrak{b}$ appears as $w_{\pm}\d t$ for $\pm s$ sufficiently large, with $w_{+}\ge w_{-}$.
\end{definition}

It is more-or-less clear how continuation cylinders of type S can be ``glued'' to pairs-of-pants of type S, provided the function $H$ is the same and the weights match at the interface. We have the following result ensuring that the resulting operations on $\mathit{HF}_{I}$ compose in a similar manner. Due to the rather subtle nature of the product operation in filtered Floer cohomology \S\ref{sec:pants-appendix}, the statement is a bit technical. Nonetheless, the statement plays a crucial role in showing that the Viterbo restriction map of \S\ref{sec:viterbo-restriction} respects the BV algebra structure.

\begin{lemma}
  Fix a smooth function $H\ge 0$. Suppose that $\mathfrak{H}$ is a Hamiltonian connection of type S on the pair of pants for $H$ with weights $(w_{0},w_{1};w_{\infty}')$. Similarly let $\mathfrak{C}_{i}$, $i=0,1,\infty$ be continuation cylinders of type S so $\mathfrak{C}_{i}$ has weights $(w_{i}';w_{i})$.

  Let $\mathfrak{H}'$ be the connection of type $S$ obtained by gluing $\mathfrak{C}_{0},\mathfrak{C}_{1},\mathfrak{C}_{\infty}$.

  Suppose that $\mathit{Spec}(wH)$ does not intersect $\set{0}\cup [\tau w,2\tau w]$ for some number $\tau>0$ for $w=w_{i}$ or $w=w_{i}'$, $i=0,1$, and let $I=[0,\tau]$, so $wI=[0,w\tau]$. Then the operations:
  \begin{itemize}
  \item $\mu_{\mathfrak{H}}:\mathit{HF}_{w_{0}I}(w_{0}H)\otimes \mathit{HF}_{w_{1}I}(w_{1}H)\to \mathit{HF}_{w_{\infty}'I}(w_{\infty}'H),$
  \item $\mathfrak{c}_{\mathfrak{C}_{i}}:\mathit{HF}_{w_{i}'I}(w_{i}'H)\to \mathit{HF}_{w_{i}I}(w_{i}H)$,
  \item $\mu_{\mathfrak{H}'}:\mathit{HF}_{w_{0}'I}(w_{0}'H)\otimes \mathit{HF}_{w_{1}'I}(w_{1}'H)\to \mathit{HF}_{w_{\infty}I}(w_{\infty}H)$,
  \end{itemize}
  are all well-defined following \S\ref{sec:pants-appendix}, \S\ref{sec:continuation-maps}, and it holds that:
  \begin{equation}\label{eq:identity-holds}
    \mu_{\mathfrak{H}'}=\mathfrak{c}_{\mathfrak{C}_{\infty}}\circ \mu_{\mathfrak{H}}\circ (\mathfrak{c}_{\mathfrak{C}_{0}}\otimes \mathfrak{c}_{\mathfrak{C}_{1}}).
  \end{equation}
\end{lemma}
\begin{proof}
  That the morphisms are well-defined is a special case of the general results of \S\ref{sec:pants-appendix} and \S\ref{sec:continuation-maps}; that the identity \eqref{eq:identity-holds} follows from standard gluing results, the fact that the curvature term of all connections of type S is zero (Lemma \ref{lemma:curvature-split}), and the same arguments used in the proof of Lemma \ref{lemma:window-spacing} (this uses the fact that $[\tau w,2\tau w]$ never intersects the spectrum). The last part about the commutativity with the BV operators follows from \S\ref{sec:rotationally-symmetric}, and the fact that any connection of type S on the cylinder is homotopic through connections of type S to a rotationally symmetric one.
\end{proof}

\subsubsection{Connections of type S, PSS elements, and the BV operator}
\label{sec:connections-type-s}

In this section we continue with the special class of connections of type S introduced in \S\ref{sec:conn-pair-pants}.

Fix a connection $\mathfrak{H}$ of type S on the cylinder for smooth function $H$ with weights $(0,w')$ with $w'>0$; this is a special case of Definition \ref{definition:type-S-cyl}. Given $f:C\to W$, we can apply the construction of \S\ref{sec:pss-elements} to produce an element:
\begin{equation*}
  \mathit{PSS}(\mathfrak{H},f,J)\in \mathit{HF}_{w'I}(w'H)
\end{equation*}
provided that $I=[0,\tau]$ and $\set{0,\tau w'}\not\in \mathit{Spec}(w'H)$. We have:
\begin{lemma}
  Let $\mathfrak{H},H,w',I,f$ be as above. The PSS element $\mathit{PSS}(\mathfrak{H},f,J)$ depends only on $H,w',I,f$, and not on the precise connection $\mathfrak{H}$.
\end{lemma}
\begin{proof}
  This is because the space of one-forms $\beta$ satisfying $\d\beta\le 0$ is convex, and so one can apply the last sentence of Lemma \ref{lemma:functorial-pss}.
\end{proof}
These type S PSS elements are compatible with type S continuation maps:
\begin{lemma}
  Let $\mathfrak{H}$, $H$, $w'$, and $f$ be as above. If $\mathfrak{C}$ is a continuation cylinder of type S for $H$ with weights $w',w$, such that $\mathit{Spec}(wH)$ is disjoint from $\set{0,\tau w}$, and similarly for $w'$, then,
  \begin{equation*}
    \mathfrak{c}_{\mathfrak{C}}(\mathit{PSS}(\mathfrak{H},f,J))=\mathit{PSS}(\mathfrak{H}',f,J)
  \end{equation*}
  as elements of $\mathit{HF}_{I}(cH)$, where $\mathfrak{H}'$ is obtained by gluing $\mathfrak{H}$ and $\mathfrak{C}$.
\end{lemma}
\begin{proof}
  The argument is a standard homotopy argument, similar to the arguments given in the preceding sections.
\end{proof}

\begin{remark}\label{remark:BV-commute-type-S}
  We will also require some version of the statement that continuation maps of type S commute with the BV operator; however, the necessary statement will follow from \S\ref{sec:rotationally-symmetric}, since any connection of type S on the cylinder is homotopic through connections of type S to one which is rotationally symmetric.
\end{remark}

\bibliographystyle{../amsalpha-doi}
\bibliography{c}
\end{document}